\documentclass[11pt,letterpaper]{amsart}

\usepackage[utf8]{inputenc}  
\usepackage{dsfont}
\usepackage{enumerate}          
\usepackage{amssymb,amsmath,amsfonts,amsthm}  
\usepackage{anysize}

\usepackage{hyperref}
\usepackage{xcolor}

{\theoremstyle{plain} 
\newtheorem{theorem}{Theorem}[section]

\newtheorem{proposition}[theorem]{Proposition}
\newtheorem{lemma}[theorem]{Lemma}

\newtheorem{remark}[theorem]{Remark}
}
\renewcommand{\S}{\mathcal{T}}

\newcommand{\R}{\mathbb{R}}
\newcommand{\Z}{\mathbb{Z}}
\newcommand{\N}{\mathbb{N}}
\newcommand{\supp}{{\rm supp}}
\newcommand{\vol}{\text{vol}}

\newcommand{\Vol}{\text{Vol}}
\newcommand{\net}{\mathcal{N}}
\newcommand{\hyp}{\widetilde{M}}
\newcommand{\cutheat}{C}
\newcommand{\cutchi}{C'}
\renewcommand{\sf}{G(M)}

\title[Maximal multiplicity of Laplacian eigenvalues]
{Maximal multiplicity of Laplacian eigenvalues in negatively curved surfaces}
 \author[C. Letrouit]{Cyril Letrouit}
\address[C. Letrouit]{Laboratoire de Mathématiques d’Orsay, Université Paris-Saclay, Bâtiment 307, 91405 Orsay Cedex \& CNRS UMR 8628, France}
\email{cyril.letrouit@universite-paris-saclay.fr}
\author[S. Machado]{Simon Machado}
\address[S. Machado]{ETH Zürich, Department of Mathematics, Rämistrasse 101, 8092, Zürich, Switzerland}
\email{smachado@ethz.ch}

\begin{document}    
\date{\today}
\maketitle

\begin{abstract}
In this work, we obtain the first upper bound on the multiplicity of Laplacian eigenvalues for negatively curved surfaces which is sublinear in the genus $g$. Our proof relies on a trace argument for the heat kernel, and on the idea of leveraging an $r$-net in the surface to control this trace. This last idea was introduced in \cite[Theorem 2.2]{equiangular} for similar spectral purposes in the context of graphs of bounded degree. Our method is robust enough to also yield an upper bound on the ``approximate multiplicity" of eigenvalues, i.e., the number of eigenvalues in windows of size $1/\log^\beta(g)$, $\beta>0$. 
This work provides new insights on a conjecture by Colin de Verdière \cite{CdV86} and new ways to transfer spectral results from graphs to surfaces. 
\end{abstract}

\section{Introduction}\label{s:intro}
\subsection{Main results}\label{s:mainresults}

Let $M$ be a closed, connected Riemannian manifold, and let $\Delta$ denote the Laplace-Beltrami operator on $M$, simply called ``Laplacian" in the sequel, which is self-adjoint and non-positive. The operator $-\Delta$ has a discrete spectrum 
\begin{equation}\label{e:eigenvaluesM}
0=\lambda_1(M)< \lambda_2(M)\leq \ldots \rightarrow +\infty,
\end{equation}
where the $\lambda_i(M)$ are repeated according to their multiplicity.

Our first main result deals with the case where $M$ is a closed negatively curved surface. We denote by $\S$ the set of triples 
$$
\S=\{(a,b,\rho)\in\R^3 \mid b\leq a<0, \ \rho>0\}.
$$
For any $(a,b,\rho)\in\S$, let $\mathcal{M}_{g}^{(a,b,\rho)}$ be the set of closed connected  surfaces of genus $g$, with injectivity radius $\geq \rho$, and with Gaussian curvature $c(x)$ satisfying $b\leq c(x)\leq a$ for any $x\in M$. An important example is obtained by taking $(a,b,\rho)=(-1,-1,\rho)$, in which case $\mathcal{M}_g^{(a,b,\rho)}$ is the set of hyperbolic surfaces (i.e., with constant curvature $-1$) of injectivity radius $\geq\rho$.

We obtain a general sublinear upper bound on the maximal multiplicity of $\lambda_2(M)$ for negatively curved surfaces. Our first main result is the following:
\begin{theorem}[Maximal multiplicity of $\lambda_2$]\label{t:sublinearA}
For any $(a,b,\rho)\in\S$, there exists $C_{0}>0$ such that for any $g\geq 2$ and any $M\in \mathcal{M}_{g}^{(a,b,\rho)}$, the multiplicity of $\lambda_2(M)$ is at most  $C_0\frac{g}{\log\log (1+g)}$. 

For any $(a,b,\rho)\in\S$ and $\delta>0$, there exist $C_1$, $\alpha>0$ such that for any $g\geq 2$ and any $M\in \mathcal{M}_g^{(a,b,\rho)}$ with spectral gap $\lambda_2(M)\geq \delta$, the multiplicity of $\lambda_2(M)$ is at most $C_1 \frac{g}{\log^\alpha g}$.
\end{theorem}

The dependence of $C_0, C_1, \alpha$ on $a,b,\rho, \delta$ is explicit. For instance, $C_0=C_0' \frac{|b| +\rho^{-2}}{|a|}$ with $C_0' > 0 $ universal, see Remark \ref{r:depC0}. Our strategy of proof partly relies on a geometric idea which takes its source in  \cite{equiangular}. This last work proves the same sublinear bound as ours, for the adjacency matrix of combinatorial graphs with a uniform bound on the degree.

Our next statement is stronger than Theorem \ref{t:sublinearA}, in the sense that it accommodates for ``approximate multiplicity" in a window of size $O(1/\log^\beta(g))$, $\beta>0$ (see Remark \ref{r:tracemethod} for comments on the size of this window). This result parallels a similar statement \cite[Theorem 1.6]{haiman} for graphs with a uniform bound on the degree of each vertex.
\begin{theorem}[Maximal approximate multiplicity of $\lambda_j$]\label{t:extension}
For any $j\in\N_{\geq 2}$, any $(a,b,\rho)\in\S$, and any $\beta,K>0$, there exists $C_0>0$ and $g_0\in\mathbb{N}_{\geq 2}$ such that the number of eigenvalues in $[\lambda_j(M),(1+\frac{K}{\log^\beta g})\lambda_j(M)]$ is at most  $C_0\frac{g}{\log \log g}$ for any $g\geq g_0$ and any $M\in \mathcal{M}_g^{(a,b,\rho)}$. 
\end{theorem}

Following an analogous result on regular graphs \cite[Proposition 5.3]{mckenzie} and using a graphs to surfaces transfer principle due to Colin de Verdière and Colbois \cite{ColboisCdV}, we also provide a
construction of closed hyperbolic surfaces with high approximate multiplicity. This result shows
that if the injectivity radius is allowed to tend to $0$, no bound like the one in Theorem \ref{t:extension} can hold. 

\begin{proposition}\label{p:construction}
For any sequence of positive numbers $(\varepsilon_g)_{g\in\mathbb{N}_{\geq 2}}$, there exists a family of connected closed hyperbolic surfaces $(M_g)_{g\in\mathbb{N}_{\geq 2}}$, where $M_g$ has genus $g$,  with at least $g-1$ eigenvalues in $[\lambda_2(M_g),(1+\varepsilon_g)\lambda_2(M_g)]$.
\end{proposition}

\subsection{A general volume-dependent bound}
Theorems \ref{t:sublinearA} and \ref{t:extension} are deduced from more general statements, which bound the (approximate) multiplicity in terms of three geometric quantities: the volume, the injectivity radius, and a lower bound on the Gaussian curvature.  For instance:
\begin{theorem}\label{t:volumebound} For any $\rho>0$ and $b<0$ there exists $C_0>0$ such that for any closed, connected Riemannian surface $M$ with injectivity radius ${\rm inj}(M)\geq \rho$ and Gaussian curvature $\geq b$, the multiplicity of $\lambda_2(M)$ is at most $C_0(1+\frac{{\rm vol}(M)}{\log\log (3+{\rm vol}(M))})$.
\end{theorem}
If a negative upper bound on the Gaussian curvature is imposed in addition, then together with the Gauss-Bonnet theorem this implies Theorem \ref{t:sublinearA}.
In complement to Theorem \ref{t:volumebound}, we refer the reader to Theorem \ref{t:volboundspecgap}, which is a version of Theorem \ref{t:volumebound} with spectral gap, to Theorem \ref{t:volumebound2}, which accommodates for approximate multiplicity and works for any $\lambda_j$, and to Theorem \ref{t:sublinearA2}, which is a scale-free version of Theorem \ref{t:volumebound}.

\subsection{Bibliographical comments}\label{s:biblio}
The maximal multiplicity of Laplacian eigenvalues has been studied at least since the 1970's and a seminal paper of Cheng \cite{cheng}. We review the literature, mostly focusing on the case of surfaces (see however the last paragraph of this section on higher dimension). For $M$ a closed surface of genus $g$, let $m_i(M)$ denote the maximal multiplicity of the $i$-th eigenvalue of a Riemannian Laplacian on $M$ (with the convention \eqref{e:eigenvaluesM} on indexing of eigenvalues).

\medskip

\noindent\textbf{Linear bounds.} Cheng proved in \cite{cheng}  that $m_i(M)\leq \frac12 (2g+i)(2g+i+1)$. This result has been improved by Besson in \cite{besson} who sharpened the bound down to $4g+2i-1$. Both papers proceed by bounding the order of vanishing of eigenfunctions and obtaining a contradiction if an eigenspace is too large (see also \cite[Section III.6]{schoen}). Then, Sévennec \cite{sevennec} proved that in negative Euler characteristic, $m_2(M)\leq 5-\chi(M)$; in particular, if $M$ is orientable of genus $g\geq 2$, then $m_2(M) \leq 2g + 3$. This bound has been improved to $2g-1$ for closed hyperbolic surfaces of sufficiently high genus in \cite[Theorem 9.5]{fortier2}.

\medskip

\noindent\textbf{Sublinear bounds.} Some sublinear bounds on the multiplicity of eigenvalues on surfaces are already available in the literature. However these bounds work only under two strong assumptions: first, they hold for \emph{hyperbolic} surfaces only and second, they require some control (i) either over the number of closed geodesics of length $\leq c\log(g)$ for $c>0$ a small constant,  (ii) or, for any $L>0$, over the number of closed geodesics of length $\leq L$ as $g\rightarrow +\infty$. This control is related to the notion of \emph{Benjamini-Schramm convergence} (see \cite{samourais}), satisfied in particular with high probability by hyperbolic surfaces drawn with respect to the Weil-Petersson probability measure.

The assumption of hyperbolicity allows to use tools which are not available for general negatively curved surfaces: for instance, the Selberg trace formula which relates, in closed hyperbolic surfaces, the spectrum of the Laplacian to the set of lengths of closed geodesics. This formula has been used in \cite[Theorem 1.6]{laura} to derive a sublinear bound on the multiplicity of Laplacian eigenvalues for random hyperbolic surfaces (see \cite[Corollary 1.7]{lemasson} for a different but related sublinear bound using the Selberg transform). Relying also on the Selberg trace formula, \cite[Proposition 9.3]{fortier2} gives a sublinear bound when $\lambda_2(M)-\frac14$ is of order $1/\log(g)^2$. Finally, we mention \cite{samourais} which proves sublinear bounds for the related problem of limit multiplicities under Benjamini--Schramm convergence, and the papers \cite{degeorge}, \cite{sarnakxue} and \cite{gamburd} which give precise rates of convergence - with power saving - under more restrictive arithmetic assumptions.

Our method, which works for any surface and does not assume a control over the number of closed geodesics of length $\leq L$ for large $L$, is totally different. It relies mainly on an ingredient inspired from the work \cite{equiangular} pertaining to multiplicities in graphs (see Section \ref{s:strategy}), and on heat kernel estimates, which correspond geometrically to random walks and not to closed geodesics.  

We mention the fact that our results yield a sublinear bound on multiplicity when restricting to the set of Riemannian covers of a fixed negatively curved surface. Also, when $a=b=-1$ and $\rho$ is small, we see that the set of hyperbolic surfaces considered in Theorem \ref{t:sublinearA} covers most of the moduli space of closed hyperbolic surfaces of genus $g\geq 2$ since the event of having injectivity radius $\geq \rho$ has probability roughly $1-\rho^2$ for the Weil-Petersson probability measure (see \cite[Theorem 4.1]{mirzakhanipetri}).

\medskip

\noindent\textbf{Colin de Verdière's conjecture.} Colin de Verdière conjectures in  \cite[Section V]{CdV86} a much stronger bound of order $\sqrt{g}$ for the maximal multiplicity. More precisely, he conjectures that
\begin{equation}\label{e:conjcdv}
m_2(M) = \text{chr}(M)-1
\end{equation}
where $\text{chr}(M)$ is the chromatic number of $M$, defined as the largest $n$ such that the complete graph on $n$ vertices embeds in $M$. By a result of Ringel
and Youngs \cite{ringel},
$$
\text{chr}(M)=\Bigl\lfloor \frac12\left(7+\sqrt{49-24\chi(M)}\right)\Bigr\rfloor,
$$
(if $M$ is not the Klein bottle, for which $\text{chr}(M)=6$), and since $\chi(M)=2-2g$ for closed orientable surfaces, $m_2(g)$ would be of order $\sqrt{12g}$.

Although a clear improvement over the linear bound of Cheng and Besson, our sublinear bound $m_2(M)\lesssim g/\log\log(g)$ is far from the conjectured $\sqrt{12g}$. However, since the first version of this paper appeared, the Colin de Verdière conjecture has been disproven for genus $g=10$ and $g=17$ \cite{counter}. It is unclear whether these are isolated cases or indicative of a much more general failure of the conjecture. Precisely, is the right order of magnitude for $m_2(M)$ closer to $\sqrt{g}$ or $g/\log\log g$?

This conjecture was mostly supported by two \emph{lower bounds}: first, Colbois and Colin de Verdière constructed in \cite{ColboisCdV} for any $g\geq 3$ a closed hyperbolic surface of genus $g$ such that the multiplicity of $\lambda_2$ is $\Bigl\lfloor \frac{1+\sqrt{8g+1}}{2}\Bigr\rfloor$, which has the same order of growth as the conjectured upper bound \eqref{e:conjcdv}. Secondly, it is proved in \cite[Théorème 1.5]{CdVENS} that $\overline{m}_2(M)\geq \text{chr}(M)-1$ for arbitrary $M$, if $\overline{m}_2(M)$ denotes the maximal multiplicity of $\lambda_2$, where the maximum is taken over all Schrödinger operator on $M$ for which $\lambda_1=0$. 

The Colin de Verdière conjecture was also supported by the fact that \eqref{e:conjcdv} is satisfied for simple choices of $M$: the sphere \cite{cheng}, the torus \cite{besson}, the projective plane \cite{besson}, the Klein bottle \cite{CdVENS}, \cite{nadirashvili}. The work \cite{fortier1} shows that the Klein quartic maximizes the multiplicity of $\lambda_2$ among all closed hyperbolic surfaces of genus 3, with multiplicity equal to 8, which also matches the equality \eqref{e:conjcdv}. The proof is based on the Selberg trace formula.

 Proving \emph{upper bounds} on the multiplicity seems much more challenging. Proposition \ref{p:construction} highlights some of the challenges. For instance, extending Theorem \ref{t:extension} to surfaces with small injectivity radius would require a proof which separates between eigenvalues very close to one another, since there exist hyperbolic surfaces $M$ of (large) genus $g$ with ``approximate multiplicity" of $\lambda_2(M)$ of order $g$.

\medskip

\noindent\textbf{Literature on graphs.}
As already mentioned in Section \ref{s:mainresults}, our inspiration comes from the following result proved in \cite[Theorem 2.2]{equiangular}:
\begin{theorem}[\cite{equiangular}]\label{t:equiangular}
For every $j$ and every $d$, there is a constant $C=C(d,j)$ so that the adjacency matrix of every connected $n$-vertex graph with maximum degree at most $d$ has $j$-th eigenvalue multiplicity at most $Cn/\log\log n$.
\end{theorem}

The main motivation of the authors of \cite{equiangular} is the equiangular problem, namely the computation of the maximal number of lines in $\R^d$ which are pairwise separated by the same angle. This problem shows up for instance through tight frames in coding theory. In \cite[Theorem 1.2]{equiangular}, the equiangular problem for a fixed angle $\alpha$ between the lines is solved by showing that it may be reduced to Theorem \ref{t:equiangular}.

We also mention the work \cite{mckenzie}, in which an improvement of \cite[Theorem 2.2]{equiangular} is proven for regular graphs. This improvement does not seem easy to transfer to (negatively curved) surfaces. 

\medskip

\noindent\textbf{Higher dimension.}
On any closed manifold $M$ of dimension $n \geq 3$, it is possible to construct a sequence of metrics whose first non-trivial eigenvalue multiplicity tends to $+\infty$ (see \cite{CdV86}). On the other hand, it is proved in \cite[Corollary 1.1]{hassan} that for $n$-dimensional manifolds whose geometry is controlled, the multiplicity of $\lambda_2$ is bounded: there exists $C$ depending on $n$ only such that in any $n$-dimensional Riemannian manifold with ${\rm Ricci}\geq -(n-1)\kappa$ for some $\kappa\geq 0$, the multiplicity of $\lambda_2$ is at most $C_0(1+{\rm vol}(M)(\kappa^{n/2}+{\rm inj}(M)^{-n}))$ where ${\rm inj}(M)$ denotes the injectivity radius of $M$. This has to be compared with our Theorem \ref{t:sublinearA2}. Indeed, we believe that our proofs of Theorems \ref{t:volumebound} and \ref{t:sublinearA2} work in any dimension, up to changing constants depending on $n$. We do not pursue this here since this paper is mostly devoted to surfaces.

\subsection{Strategy of proof}\label{s:strategy} 

\subsubsection{Warm-up: proof in the graph case}\label{s:subgraph}
Our strategy to prove Theorems \ref{t:sublinearA} and \ref{t:extension} is partly inspired by the proof of Theorem \ref{t:equiangular} worked out in \cite{equiangular}.
We provide here a summary of this proof.

Let $d>0$ and let $G$ be a connected graph with degree $\leq d$, whose adjacency matrix is denoted by $A_G$. The authors of \cite{equiangular} introduce a subgraph $H\subset G$ whose complement $G\setminus H$ is an $r_1$-net: it means that any vertex of $G$ is at distance at most $r_1$ from $G\setminus H$. The parameter $r_1$ is chosen as $r_1=\lfloor c\log\log(n)\rfloor$, where $n$ is the number of vertices of $G$ and $c>0$ is a small constant.

The first step is to find an upper bound for the trace of $A_H^{2r_1}$, where $A_H$ is the adjacency matrix of $H$.
For this, the authors of \cite{equiangular} leverage the usual technique of expressing a trace as a number of closed paths. The trace $\text{Tr}(A_H^{2r_1})$ is bounded above by the number of paths of length $2r_1$ in $G$, which start from a given vertex $x\in H$ and do not belong to the $r_1$-net $G\setminus H$ at time $r_1$.  It follows from the definition of an $r_1$-net that this number is smaller by at least $1$ than the total number of paths of length $2r_1$ in $G$ which start from $x$ and end at $x$: we call this the ``gain of $1$".

This gain of $1$ is transformed into a larger gain by considering the trace of $A_H^{2r_2}$ with $r_2=\lfloor c\log(n)\rfloor\gg r_1$, instead of the trace of $A_H^{2r_1}$. The argument to get this larger gain relies on the Perron-Frobenius theorem and the min-max principle applied locally in balls of radius $r_2$. The large gain which is obtained provides a strong bound on $\text{Tr}(A_H^{2r_2})$, and thus on the number of eigenvalues of $A_H$ close to $\lambda_2(A_H)$ (recall that $\lambda_1(A_H)>\lambda_2(A_H)\geq \ldots$). 

Finally, the Cauchy interlacing theorem (Theorem \ref{t:cauchy}) converts this bound into a similar bound on the eigenvalues of $A_G$. The bounds depend on $d$.

\subsubsection{Main steps}
The main steps of our proof of Theorem \ref{t:volumebound} mimic the above proof, with many additional difficulties and several new ideas. We focus on the case where ${\rm vol}(M)\gg1$, since the small volume case will be seen to follow from \cite{hassan}.
\begin{enumerate}[(i)]
    \item We consider, instead of directly $\lambda_2=\lambda_2(M)$, the maximal multiplicity of $e^{-t\lambda_2(M)}$ as an eigenvalue of $e^{t\Delta}$, thus reinterpreting the problem in terms of heat kernels (and random walks).
    \item In analogy with the graph case, we set $r_1=c\log\log({\rm vol}(M))$ where  
    $c>0$ is a small constant. In the setting of Theorem \ref{t:sublinearA}, $r_1\approx c\log\log g$, and the genus $g$ plays the role of the number of vertices $n$ in the graph case. We choose an $r_1$-net $\{x_1,\ldots,x_\ell\}\subset M$: this means that any point in $M$ is at distance at most $r_1$ from one of the $x_k$'s. Then we fix around each $x_k$ an open set $V_k$ of measure $\sim 1$. We define the operator $P:L^2(M,\nu)\rightarrow L^2(M,\nu)$ as the orthogonal projection to the space of functions which are $L^2$-orthogonal to the (normalized) characteristic functions of the $V_k$'s. 
    \item We use a Cauchy interlacing theorem in Hilbert spaces (see Theorem \ref{t:cauchy}): we compare the  multiplicity $m$ of $e^{-r_1\lambda_2}$ as an eigenvalue of $e^{r_1\Delta}$ with the multiplicity $m'$ of $e^{-r_1\lambda_2}$ as an eigenvalue of $Pe^{r_1\Delta}P$. 

The Cauchy interlacing theorem implies that 
    \begin{equation}\label{e:mm'P}
    m\leq m'+\text{rank}(\text{Id}-P).
    \end{equation}
    Our choice of $P$ guarantees that $\text{rank}(\text{Id}-P)=O({\rm vol}(M)/\log \log({\rm vol}(M)))$, or even $\text{rank}(\text{Id}-P)=O({\rm vol}(M)/\log^\alpha({\rm vol}(M)))$  when we work under the assumption $\lambda_2(M)\geq \delta$ (to prove the second part of Theorem \ref{t:sublinearA}). The next steps prove an upper bound on $m'$.
    \item We choose $r_2=c\log {\rm vol}(M)$ and $n\approx \lfloor r_2/r_1\rfloor$ and
    we compute the trace of  $(Pe^{r_1\Delta}P)^{2n}$ to bound above $m'$:
    \begin{equation}\label{e:introstep1}
m'e^{-2nr_1\lambda_2}\leq \text{Tr}((Pe^{r_1\Delta}P)^{2n}).
    \end{equation}
    The trace in the right-hand side may in turn be written as an integral of the form 
    \begin{equation}\label{e:introstep2}
   \text{Tr}((Pe^{r_1\Delta}P)^{2n})=\int_M \|(Pe^{r_1\Delta}P)^{n}\delta_x\|^2d\nu(x).
    \end{equation}
    \item We leverage the averaging properties of the heat kernel to prove an inequality which roughly looks like\footnote{Here we warn the reader that the sequence of inequalities we prove is actually much more subtle than \eqref{e:rough}.}
    \begin{equation}\label{e:rough}
    \|(Pe^{r_1\Delta}P)^{\lfloor r_2/r_1\rfloor}\delta_x\|\leq \left((1-\varepsilon)e^{-r_1\lambda_2}\right)^{\lfloor r_2/r_1\rfloor}
    \end{equation}
    for ``most points" $x\in M$. The assumption on the injectivity radius in Theorem \ref{t:volumebound} comes from the proof of \eqref{e:rough}, but also from the construction of the $r_1$-net in Step (ii). \\
    Combining \eqref{e:introstep1}, \eqref{e:introstep2} and \eqref{e:rough} we obtain for some $C_0>0$
    \begin{equation}\label{e:m'upperbound}
    m'\leq C_0 {\rm vol}(M)(1-\varepsilon)^{2\lfloor r_2/r_1\rfloor}\leq C_0 {\rm vol}(M)\exp\left(-2\varepsilon\lfloor r_2/r_1\rfloor\right)
    \end{equation}
    The quantity $\varepsilon>0$, which depends on the volume ${\rm vol}(M)$, is the ``gain", and we prove it to be sufficiently large, so that Theorem \ref{t:volumebound} follows from \eqref{e:mm'P}, \eqref{e:m'upperbound} and our choices of $r_1$, $r_2$, $P$. 
\end{enumerate} 
The proof of Step (v) is the heart of our contribution, and a more detailed summary of this step is provided at the beginning of Section \ref{s:gainestimates}, before its actual proof. Whereas the ``gain" is straightforward to obtain for graphs (see Section \ref{s:subgraph}), we have to face in the case of surfaces several difficulties.

A first difficulty comes from the infinite speed of propagation of the heat kernel\footnote{Although there exists a ``random walk at speed $1$" on manifolds (see for instance \cite{lebeaumichel}), whose kernel is the most obvious analogue of the adjacency matrix $A_G$ of Section \ref{s:subgraph}, we use in this paper the heat kernel because it seems more natural to understand the Laplacian on manifolds, and because the bounds available on the kernel of the ``random walk at speed $1$" are not as good as the ones available on the heat kernel.}. This property a priori prevents us from using any local argument in the manifold; however, as mentioned in Section \ref{s:subgraph}, we need to apply the min-max principle locally in balls of radius $\approx r_2$ to obtain the quantitative gain $\varepsilon$. To overcome this difficulty we introduce some cut-offs $\chi_x$ (approximately the characteristic function of a ball of center $x$ and radius $\cutchi r_2$ for some large $\cutchi$) commuting with $P$, and consider the compact operators $B_x=P\chi_xe^{r_1\Delta}\chi_x P$ instead of $Pe^{r_1\Delta}P$ in \eqref{e:rough}. The remainder terms which unavoidably appear when replacing $Pe^{r_1\Delta}P$ by $B_x$ are handled through classical heat kernel estimates in the universal cover of $M$.

Another difficulty arises from the fact that the operator $B_x$, which somehow plays locally around $x$ the role of $A_H^{r_1}$ in Section \ref{s:subgraph}, has one main difference with $A_H^{r_1}$: its matrix elements are not necessarily non-negative (the condition $f,g\geq 0$ does not imply that $(B_xf,g)\geq0$), and the Perron-Frobenius theorem therefore does not apply to $B_x$; however, as mentioned in Section \ref{s:subgraph}, we do need to apply a Perron-Frobenius-type argument in local balls. We overcome this difficulty by analyzing the interplay between the positive and the negative part of the top eigenvector  $\varphi_x$ of $B_x$. This allows us to recover a gain $\varepsilon$ despite the lack of positivity (in the sense of matrix elements) of $B_x$.

\medskip

\noindent\textbf{Organization of the paper.} The paper is organized as follows. We introduce useful notation in Section \ref{s:notation}, and we prove elementary results regarding $r$-nets in Section \ref{s:voronoi}. In Section \ref{s:heathyp} we state estimates on the heat kernel in $M$ and its universal cover $\widetilde{M}$. Section \ref{s:key} gathers the key lemmas used in the proof of Theorem \ref{t:volumebound}: in Section \ref{s:errorterms} we compare the trace $\text{Tr}((Pe^{r_1\Delta}P)^n)$ to an integral of local Rayleigh quotients and we estimate the error terms; in Section \ref{s:applicationsminmax} we draw several consequences from the min-max principle used to bound the previously mentioned Rayleigh quotients; in Section \ref{s:gainestimates}, we prove the gain described in Step (v) above.  In Section \ref{s:proofofth}, we proceed with the proof of Theorems \ref{t:volumebound} and \ref{t:sublinearA}, and in Section \ref{s:proofofextension} we explain how to modify this proof to obtain Theorem \ref{t:extension}. In Section \ref{s:scalefree} we prove version of Theorem \ref{t:volumebound} involving only quantities which are invariant under rescaling of the metric (``scale-free"). In Section \ref{s:construction} we prove Proposition \ref{p:construction}, relying on a construction due to \cite{ColboisCdV}.
%and \cite{mckenzie}. 
In Appendix \ref{a:eigenvalues}, we gather several elementary results such as the Cauchy interlacing theorem in infinite dimensional Hilbert spaces and an upper bound for the first eigenvalues of closed negatively curved surfaces. Finally in Section \ref{a:heatkernel} we prove the heat kernel estimates stated in Section \ref{s:heathyp}.

\medskip

\noindent\textbf{Acknowledgment.} The authors are thankful to Yves Colin de Verdière, Gilles Courtois, Maxime Fortier Bourque, Laura Monk, Carl Schildkraut and Yufei Zhao for answering questions related to this work. They also would like to thank the referees for their careful reading of the manuscript and their suggestions which greatly improved the paper. Part of this work was done while C.L. was supported by the Simons Foundation Grant
601948, DJ. S.M. was supported by the National Science Foundation under Grant No. DMS-1926686.

%; the isometry is denoted by $\iota:B(p_0,\text{arcsinh}(1))\rightarrow D$. 
%We fix $h$ supported in $D$ such that $h\neq 0$ and $\int_M h(x)d\Vol_{\hyp}(x)=0$. We consider $f\in C^\infty(M)$ defined by $f(x)=h(\iota(x))$ for $x\in B(p_0,\text{arcsinh}(1))$, and $f(x)=0$ otherwise. Then $\int_M f(x)d\nu(x)=0$ and  \begin{align*} \langle Af,f\rangle_{L^2(M,\nu)}&\geq \frac{1}{\Vol_{\hyp}(B_{\hyp}(1))^2}\langle A(h\circ\iota),(h\circ\iota)\rangle_{L^2(M,\nu)}\\ &\geq C_\delta\frac{1}{\Vol_{\hyp}(B_{\hyp}(1))^2}\langle A(h\circ\iota),(h\circ\iota)\rangle_{L^2(M,\Vol)}\\ &=C_\delta\frac{1}{\Vol_{\hyp}(B_{\hyp}(1))^2}\langle \widetilde{A}h,h\rangle_{L^2(\hyp,\Vol_{\hyp})} \end{align*} where $\widetilde{A}=\iota A \iota^{-1}:L^2(D,\Vol_{\hyp})\rightarrow L^2(D,\Vol_{\hyp})$ and  $\iota$ sends $L^2(B(p_0,\text{arcsinh}(1)),\Vol)$ to $L^2(D,\Vol_{\hyp})$ by abuse of notations.  Since we also have $$ \|f\|_{L^2(M,\nu)}^2\geq C_\delta'\|h\|^2_{L^2(M,\Vol)}$$ we deduce that $\mu_2\geq C_\delta''$. 

\section{Preliminaries}\label{s:preliminaries}
This section gathers notation and elementary results concerning Voronoi cells in closed  surfaces, $r$-nets and the heat kernel. We work in the general setting of Theorem \ref{t:volumebound}, therefore we fix $\rho>0$, $b<0$ and $M$ a closed, connected Riemannian surface with injectivity radius ${\rm inj}(M)\geq \rho$ and Gaussian curvature $c(x)\geq b$ for any $x\in M$. We also assume that ${\rm vol}(M)\geq 3$, in order for $\log\log({\vol}(M))$ to be well-defined; the case of small volumes will be handled separately at the end of Section \ref{s:proofofth}.

\subsection{Notation}\label{s:notation}
%, and it is assumed that its Gaussian curvature $K$ satisfies: \begin{equation}\label{e:boundscurvature}-b\leq K\leq -a\end{equation}
 The Riemannian distance in $M$ is denoted by $d(\cdot,\cdot)$ and the open ball of center $x\in M$ and radius $r>0$ is $B_d(x,r)$. 

We denote by $\nu$ the Riemannian volume on $M$, by $\langle \cdot,\cdot\rangle$ the scalar product with respect to $\nu$, and by $\|\cdot\|$ the associated norm. We introduce the positive parameters
\begin{equation}\label{e:defr1r2}
r_1=c\log\log({\rm vol}(M)), \quad r_2= c\log{\rm vol}(M)
\end{equation}
where $c>0$ is a small constant which will be fixed in Section \ref{s:proofofth} (see Remark \ref{r:choiceparameter} for an explanation on the choice of these parameters).

Let $(\hyp,d_{\hyp})$ be the universal cover of $M$ endowed with the lifted Riemannian metric.  Let $\Vol_{\hyp}$ be the Riemannian volume on $\hyp$. We recall that the volume of any ball $B(x,r)$ of radius $r$ in $\hyp$ satisfies 
\begin{equation}\label{e:volumeballmodel}
\Vol_{\hyp}(B(x,r))\leq \frac{4\pi}{|b|}\sinh^2\left(\frac{\sqrt{|b|}}{2}r\right)
\end{equation}
according to the Bishop-Gromov inequality.

\medskip

\noindent\textbf{Heat kernels.} We denote by $k_t:\hyp\times\hyp\rightarrow \R$ the heat kernel in $\hyp$, so that for any $f\in L^2(\hyp,d\Vol_{\hyp})$, the solution $u:\R^+\times \hyp\rightarrow\R$ of $\partial_t u=\Delta u$ with initial datum $u(0,\cdot)=f$ is given by
$$
u(t,x)=\int_{\hyp}k_t(x,y)f(y)d\Vol_{\hyp}(y).
$$
By a slight abuse, in the case where $\hyp$ has constant curvature, we use the same notation $k_t$ (with only one argument) for the function $k_t:\R^+\rightarrow \R$ defined by $k_t(d_{\hyp}(x,y))=k_t(x,y)$ for any $x,y\in\hyp$. This definition makes sense since $k_t(x,y)$ only depends on $d_{\hyp}(x,y)$.

 The linear operator $e^{\Delta}$ is compact, self-adjoint in $L^2(M,\nu)$, with norm $1$.
We denote by $K_t(x,y)$ the heat kernel on $M$, i.e., it satisfies for any $t\geq 0$
$$
e^{t\Delta}f(x)=\int_M K_t(x,y)f(y)\nu(dy).
$$ 
We have $e^{t\Delta}\delta_x=K_t(x,\cdot)\in L^2(M,\nu)$ for any $t>0$ and $x\in M$. Writing $M=\Gamma\backslash\hyp$, we have the formula 
\begin{equation}\label{e:heatkernelinM}
K_t(x,y)=\sum_{\gamma\in\Gamma}k_t(\bar{x},\gamma\bar{y})
\end{equation}
where $\bar{x},\bar{y}$ are lifts of $x,y$ in a fixed fundamental domain of $M$ in $\hyp$ (the convergence of this sum is proved in the proof of Lemma \ref{l:heatkernelinM} in Appendix \ref{a:heatkernel}).

By rescaling the Riemannian metric on $\mathbb{H}^2$, we create a space $\widetilde{M}_K$ of constant Gaussian curvature $K < 0$, which is a simply connected space form. The universal cover of a closed surface with constant Gaussian curvature $K$ is isometric to $\widetilde{M}_K$.

The following lower bound on the heat kernel is proved in \cite[Théorème 1]{gaveau}.
\begin{lemma}[Comparison for the heat kernel]\label{l:comparisonheat}
Let $\hyp$ be a complete and simply connected Riemannian manifold with associated distance $d_{\hyp}$ and heat kernel $k_t(\cdot,\cdot)$. Assume that its sectional curvature is bounded %above by $a$ and 
below by $b$. Then for any $x,x_0\in \hyp$ and any $t>0$,
$$
k_t^{(b)}(d_{\hyp}(x_0,x))  \leq k_t(x_0,x)
%\leq k_t^{(a)}(d_{\hyp}(x_0,x)) 
$$
where $k_t^{(K)}(\cdot)$ denotes the heat kernel (with radial variable) on $\widetilde{M}_K$.
\end{lemma}
Concerning upper bounds on the heat kernel, our main tool is the bound \eqref{e:fundaheat2} below, due to \cite[Theorem 3]{davies2} (when combined with \cite[Proposition 14]{croke}).

\medskip

\noindent\textbf{Constants.} Throughout the paper we use the following conventions to denote constants: 
\begin{itemize}
    \item we keep the same notation for constants which may change from line to line. 
    \item Constants with an integer subscript, namely $C_0$, $C_1$, $\ldots$, depend on $a$, $b$ and $\rho$ only.
    \item $C>0$ and $C'>0$ denote two sufficiently large constants whose values are fixed in the proof of Theorem \ref{t:volumebound}, in \eqref{e:choicecutchi}-\eqref{e:choicecutheat22}. $\cutheat$ and $\cutchi$ are introduced respectively in Lemma \ref{l:heatkernelesthyp} and at the beginning of Section \ref{s:errorterms}. 
    \item The constant $c>0$ introduced in \eqref{e:defr1r2} is fixed in \eqref{e:choicec} (chosen sufficiently small).
\end{itemize}

\subsection{$r$-nets and Voronoi cells}\label{s:voronoi}
As explained in Section \ref{s:strategy}, we need for our proof to consider a subset $V$ of  measure of order $1$ such that any point of $M$ is at distance $\leq r_1$ of $V$. In the context of graphs, e.g., in \cite{equiangular}, such sets, called $r_1$-nets, are subsets of the sets of vertices. In the case of a closed negatively curved surface $M$, we cut $M$ into Voronoi cells and select a (not too large) subset of cells well distributed over $M$. This section gathers the necessary definitions and results.

In the sequel, an \textit{$r$-separated set} is a set of points $x_1,\ldots,x_m\in M$ such that $d(x_i,x_j)\geq r$ for any distinct $i,j\in\{1,\ldots,m\}$. An \textit{$r$-net} is a set of points $x_1,\ldots,x_m\in M$ such that for any $y\in M$ there exists $i\in\net$ such that $d(y,x_i)\leq r$.

The following two lemmas prove the existence of $r$-nets whose size is not too large.

%\begin{definition}
%We call $r$-net in the Riemannian manifold $M$ a set of points %$\{x_1,\ldots,x_k\}$ such that for any $y\in M$ there exists %$i\in \{1;\ldots;k\}$ with $d(y,x_i)\leq r$.
%\end{definition}
\begin{lemma}\label{l:taillernet}
For any $\rho>0$ there exists $C_{0}>0$ such that for any closed, connected Riemannian surface $M$ with ${\rm inj}(M)\geq \rho$, and for any $r\geq 1$, there exists an $r$-net in $M$ of cardinal at most $\max(1,C_0 {\rm vol}(M)/r)$.
\end{lemma}

\begin{proof}
The proof follows classical arguments, see e.g. \cite[Lemma 2.2]{hassan}.
Let $\rho>0$ and $b<0$. If $\text{diam}(M)\leq r$, then there is an $r$-net of size $1$. So suppose that $\text{diam}(M)>r$. Let $x_1, \ldots, x_\ell \in M$ be an $r$-separated set of maximal cardinality in $M$. Then $X:=\{x_1, \ldots, x_\ell\}$ is an $r$-net of $M$ and $B_d(x_i,r/2) \cap B_d(x_j,r/2)=\emptyset$ for all $i \neq j$. So 
\begin{equation}\label{e:argseparation}
\ell \min_{i\in\{1;\ldots;\ell\}} \vol(B_d(x_i,r/2)) \leq \sum_{i=1}^\ell \vol(B_d(x_i,r/2)) \leq \vol(M).
\end{equation}
It thus suffices to show that for any $x \in M$, $B_d(x, r/2)$ has volume at least $C_1 r$ for some constant $C_1 > 0$ (depending on $b<0$ and $\rho>0$ only). Fix $x\in M$. Take any $y \in M$ such that $d(x,y)\geq r/2$  -- there must be at least one such $y$ since $\text{diam}(M)> r$ -- and let $\gamma: [0;r/2] \rightarrow M$ be a continuous path of minimal length from $x$ to $y$. Then the balls $B_d(\gamma(n), 1/2)$ for $n \in \{0, \ldots ,\lfloor (r-1)/2 \rfloor\}$ are pairwise disjoint and contained in $B_d(x,r/2)$. There is in addition a universal constant $C_2 > 0$ such that $\vol(B_d(z, 1/2))\geq C_2\rho^2$ for all $z\in M$ (due to \cite[Proposition 14]{croke}), and we use this for $z=\gamma(n)$, $n \in \{0, \ldots ,\lfloor (r-1)/2 \rfloor\}$. So 
\begin{equation}\label{e:volumeball}\vol(B_d(x,r/2)) \geq \left(\Bigl\lfloor \frac{r-1}{2} \Bigr\rfloor+1\right)C_2\rho^2 \geq C_2 \rho^2 \frac{r}{2}.
\end{equation}
Hence, $\ell \leq \frac{C_3}{\rho^2}\frac{{\rm vol}(M)}{r}$ for some universal constant $C_3 > 0$.
\end{proof}

\begin{lemma}\label{l:spectralgap}
For any $\rho,\delta>0$ and $b<0$ there exists $C_0 >0$, such that for any closed, connected Riemannian surface $M$ with ${\rm inj}(M)\geq \rho$ and spectral gap $\lambda_2(M)\geq\delta$, and for any $r$ such that $1\leq r\leq \frac{1}{\sqrt{|b|}}\log(|b|{\rm vol}(M)/8\pi)$, there exists an $r$-net in $M$ of cardinal at most $\max(1,C_{0} {\rm vol}(M)/e^{\delta' r})$ where $\delta'=\max(\frac{\sqrt{\delta}}{\sqrt{20}}, \frac{\delta}{4\sqrt{|b|}})$.
\end{lemma}
\begin{proof}
By the Buser inequality \cite{buseriso}, the Cheeger constant $h(M)$ verifies $\delta\leq \lambda_2(M)\leq 2h(M)\sqrt{|b|}+10h(M)^2$, hence $h(M)\geq \delta'$. Besides, for any $x\in M$ and any $r$ as in the statement, we have 
$$\vol(B_d(x,r))\leq \Vol_{\hyp}(B_{\hyp}(r))\leq \frac{4\pi}{|b|}\sinh^2\left(\sqrt{|b|}\frac{r}{2}\right)\leq \frac12 \vol(M)$$ 
where the second inequality comes from \eqref{e:volumeballmodel} and the third one from our assumption on $r$. Hence by definition of the Cheeger constant,
\begin{align*}
\frac{d\vol(B_d(x,r))}{dr}=\lim_{\varepsilon\rightarrow 0}\frac{\vol(B_d(x,r+\varepsilon))-\vol(B_d(x,r))}{\varepsilon}=|\partial B_d(x,r)|\geq \delta'\vol(B_d(x,r))
\end{align*}
which implies that $\vol(B_d(x,r))\geq \vol(B_d(x,1)) e^{\delta'r}$. But $\vol(B_d(x,1)) \geq C_1 \rho^2$ for some universal $C_1 > 0$ (due to \cite[Proposition 14]{croke}). Replacing in \eqref{e:argseparation} and \eqref{e:volumeball} we find the result.
\end{proof}

We fix a $1$-separated set $\mathcal{V}=\{v_1,\ldots,v_m\}$ of maximal cardinality $m$ in $M$, and we consider the Voronoi cells
$$
V_k=\{q\in M\mid \forall i\in \{1,\ldots,m\},\ d(q,v_k)\leq d(q,v_i)\}, \qquad k=1,\ldots,m.
$$
		By our choice of $\mathcal{V}$, there holds $B_d(v_k,1/2)\subset V_k\subset B_d(v_k,1)$ for any $k\in\{1,\ldots,m\}$. We notice that if $\widetilde{v}_k$ denotes a lift of $v_k$ to a fundamental domain of $M$ in $\hyp$,
\begin{equation}\label{e:volmin}
\frac{4\pi}{|b|}\sinh^2\left(\frac{\sqrt{|b|}}{2}\right)\geq \Vol_{\hyp}(B_{\hyp}(\widetilde{v}_k,1)))\geq \vol(V_k)\geq C_{0}\rho^2.
\end{equation}
for some universal $C_{0}>0$, where the leftmost inequality comes from \eqref{e:volumeballmodel}, and the rightmost inequality from \cite[Proposition 14]{croke}.

We fix a subset $\net \subset \{1,\ldots,m\}$ with the following properties:
\begin{itemize}
    \item $\#\net\leq C_0 {\rm vol}(M)/r$ (or  $\#\net \leq C_{0} {\rm vol}(M)/e^{\delta' r}$ if $\lambda_2(M)\geq \delta$ is assumed).
    \item There exists a family of points $(x_k)_{k\in\net}$ forming an $r_1$-net, such that for any $k\in\net$, $x_k\in V_k$.
\end{itemize}
The set $\net$ is constructed by first considering an $r_1$-net $\{x_1,\ldots,x_\ell\}$, and then for each $j\in[\ell]$, putting in $\net$ the index of (one of) the Voronoi cell(s) to which $x_j$ belongs.

For any $k\in\net$, we denote by $\psi_k$ the normalized characteristic function of the interior $\mathring{V}_k$ of $V_k$, i.e., $\psi_k(x)=\frac{1}{\sqrt{\vol(V_k)}}\mathbf{1}_{x\in \mathring{V}_k}$. It follows from \eqref{e:volmin} that $\|\psi_k\|_{L^\infty(M)}\leq \frac{C_2}{\rho}$. 
%In the sequel, $D_\rho$ denotes the maximal diameter of the $V_k$:\begin{equation}\label{e:drho}D_\rho=\max_{k\in\net}\text{diam}(V_k).\end{equation}

We denote by $P$ the orthogonal projection\footnote{The idea of considering these projections is inspired by the paper \cite{buserriem}, which shows that in a closed hyperbolic surface $M$ of genus $g$, the number of eigenvalues below $1/4$ is bounded above by $4g-2$. The proof goes by considering a triangulation of $M$ into $4g-2$ geodesic triangles, and showing that in each of these triangles, the smallest positive eigenvalue of the Neumann problem is at least $1/4$. The functions which are considered in the Rayleigh quotient minimization problem are orthogonal to the characteristic functions of the geodesic triangles.} on the orthogonal of the $\psi_k$ with respect to $\nu$:
$$
\forall f\in L^2(M,\nu), \qquad Pf=f-\sum_{k\in\net} \langle f,\psi_k\rangle \psi_k.
$$
If $\delta_x$ denotes the Dirac mass on a manifold (defined as $\delta_x(f)=f(x)$) then $P\delta_x$ is a distribution, defined as
\begin{equation}\label{e:Pdeltax}
P\delta_x=\delta_x-\sum_{k\in\net}\psi_k(x)\psi_k.
\end{equation}

\subsection{Heat kernel estimates}\label{s:heathyp}
The following lemmas  on the heat kernels in $\hyp$ and $M$ (see definitions in Section \ref{s:notation}) will be instrumental in the proof of Theorem \ref{t:volumebound}. 
The proof of these lemmas is postponed to Appendix \ref{a:heatkernel}.

Estimate 
\eqref{e:L1norm} reflects the fact that the mass of $k_t(x,\cdot)$ 
is small outside a ball of radius $Ct$ around $x$ for $C$ sufficiently large. Then, the estimate \eqref{e:variationsalpha} shows that when restricting to the interior of a ball of radius $Ct$, the heat kernel varies not too wildly over scales $\leq 4t$ (any other constant than $4$ would also work, but $4$ is the right constant for Lemma \ref{e:Ar1|f|phik}).
\begin{lemma}[Estimates on the heat kernel in $\hyp$] \label{l:heatkernelesthyp}  The following estimates hold: 
\begin{itemize}
\item \textup{($L^1$ norm outside a ball of radius $Ct$).} There exists $C_0>0$ such that for any $C\geq 8\sqrt{|b|}+4$, $t\geq 1$ and $x\in\hyp$, 
\begin{equation}\label{e:L1norm}
\|k_t(x,\cdot)\|_{L^1(\hyp\setminus B_{\hyp}(x,\cutheat t))}\leq C_0 \exp\left(|b|t-\frac{\cutheat^2t}{16}\right).
\end{equation}
\item \textup{(Variations over larger scales inside a ball of radius $Ct$).} There exists $C_0> 0$ such that for any $t \geq 1$, $\cutheat>0$, $x,y,z\in\hyp$ with $d(x,y)\leq \cutheat t$ and $|d(x,y)-d(x,z)|\leq 4t+4$, there holds
\begin{equation}\label{e:variationsalpha}  
 \frac{k_t(x,z)}{k_t(x,y)}\geq C_0\exp\left(-5(1+|b|)(1+\cutheat)t\right).
\end{equation}
\end{itemize}
\end{lemma}

The following lemma on the $L^\infty$ norm of the heat kernel in $M$ is not sharp, but it is sufficient for our purpose. A proof can be found in Appendix \ref{a:heatkernel}.
\begin{lemma}[$L^{\infty}$ norm of the heat kernel in $M$]  \label{l:heatkernelinM}
There exists $C_{1}>0$ such that for any $t\geq 1$ there holds  
\begin{equation*}\label{e: Linftynorm}
    \|K_t(\cdot,\cdot)\|_{L^\infty(M\times M)} \leq C_{1}\exp(4|b|t).
\end{equation*}
\end{lemma}

\section{Key lemmas}
\label{s:key}
This section is devoted to the proof of several lemmas which are key ingredients of the proof of Theorem \ref{t:volumebound} provided in Section \ref{s:proofofth}. As in Section \ref{s:preliminaries}, we fix $\rho>0$, $b<0$ and $M$ a closed, connected Riemannian surface with ${\rm inj}(M)\geq \rho$ and $c(x)\geq b$ for any $x\in M$. We also assume that ${\rm vol}(M)$ is sufficiently large, so that $r_1$ is well-defined and $\geq 1$.

\subsection{Error term estimates}\label{s:errorterms}
As explained in Section \ref{s:strategy}, our proof relies on finding an upper bound on the trace of the trace-class operator $(Pe^{r_1\Delta}P)^n$, and the first step is to write this trace as an integral \begin{equation}\label{e:PeP}
\text{Tr}((Pe^{r_1\Delta}P)^{2n})=\int_M \|(Pe^{r_1\Delta}P)^{n}\delta_x\|^2d\nu(x)
\end{equation}
(see Lemma \ref{l:tracedelta}). We actually choose $n=\lfloor r_2/r_1\rfloor+1$.

To leverage the gain of $\varepsilon$ that we will obtain in Section \ref{s:gainestimates} (and which is described in Step (v) of Section \ref{s:strategy}), we have to compare $\|(Pe^{r_1\Delta}P)^{n}\delta_x\|^2$ to a quantity of the form $\|(Pe^{r_1\Delta}P)^{n}\varphi\|^2$ for some well-chosen $\varphi\in L^2(M,\nu)$ depending on $x\in M$, see \eqref{e:deltaxetphiavecreste} below.
We work this out in the present section, and we provide estimates for the error terms which unavoidably appear along the way.

We denote by $\chi_x$ the indicator function of  the subset 
\begin{equation}\label{e:defcutchi}
\bigcup_{k\in \net, \ V_k \cap B_d(x,\cutchi r_2) \neq \emptyset } V_k,
\end{equation}
where $\cutchi$ will be fixed in Section \ref{s:proofofth}. This subset contains $B_d(x,\cutchi r_2)$. Besides, $[\chi_x,P]=0$ for any $x\in M$, as a consequence of the definition of $P$ and the equality $\langle \chi_x f,\psi_k\rangle \psi_k=\langle f,\chi_x\psi_k\rangle\psi_k=\langle f,\psi_k\rangle \chi_x\psi_k$, valid for any $f\in L^2(M,\nu)$ and any $k=1,\ldots,m$.

The main result of this section is the following one.
\begin{lemma}\label{l:comparedeltatophi}
There exists $C_0>0$ such that if $\cutchi \geq 32\sqrt{|b|}+16$, then for any $x\in M$,
\begin{align}
\|(Pe^{r_1\Delta}P)^{\lfloor r_2/r_1\rfloor+1}\delta_x\| \leq C_0\left(\sup_{\|\varphi\|=1}\|(P\chi_xe^{r_1\Delta}\chi_xP)^{\lfloor r_2/r_1\rfloor}\varphi\|+\exp\left(-\frac{\cutchi^2 r_2}{32}\right)\right).\label{e:deltaxetphiavecreste}
\end{align}
\end{lemma}

The proof of Lemma \ref{l:comparedeltatophi} relies on the following intermediate result.
\begin{lemma}\label{l:simoooon}
There exists $C_0>0$ such that if $\cutchi\geq 32\sqrt{|b|}+16$, then for any $x\in M$,
$$\|\left(Pe^{r_1\Delta}P\right)^{\lfloor r_2/r_1\rfloor+1}\delta_x - \left(P\chi_xe^{r_1\Delta}\chi_xP\right)^{\lfloor r_2/r_1\rfloor+1} \delta_x\| \leq C_0\exp\left(-\frac{\cutchi^2r_2}{32}\right).$$
\end{lemma}

\begin{remark}
The cut-offs $\chi_x$ are introduced to overcome the difficulty caused by the infinite speed of propagation of the heat kernel. We expect that most of the mass of $(Pe^{r_1\Delta}P)^{\lfloor r_2/r_1\rfloor+1}\delta_x$ is contained in a ball of radius $\cutchi r_2$. Lemma \ref{l:simoooon} proves that the remainder term coming from the complement of $\supp(\chi_x)$ is small.
\end{remark}

\begin{proof}[Proof of Lemma \ref{l:simoooon}]
    Write $n:=\lfloor r_2/r_1\rfloor+1$.  We first notice that the difference $\left(Pe^{r_1\Delta}P\right)^n\delta_x - \left(P\chi_xe^{r_1\Delta}\chi_xP\right)^n\delta_x$ is equal to the telescopic sum
    \begin{equation}
     \sum_{l=0}^{n-1}\left(P\chi_xe^{r_1\Delta}\right)^{n-l-1}P\left(1 - \chi_x\right)\left(e^{r_1\Delta}P\right)^{l+1}\delta_x \label{Eq: telescoping sum}
    \end{equation}
    (to show this, write $P(1-\chi_x)=P-P\chi_x$, recall that $[P,\chi_x]=0$, and after removing telescoping terms, use $\chi_x\delta_x=\delta_x$).
    
    Next, we estimate the norm of each term of \eqref{Eq: telescoping sum} individually. For $n-1 \geq l \geq 0$, the quantity 
    $$ ||\left(P\chi_xe^{r_1\Delta}\right)^{n-l-1}P\left(1 - \chi_x\right)\left(e^{r_1\Delta}P\right)^{l+1}\delta_x||$$
    is bounded above by 
    \begin{equation} \label{e:onveutunebornesup}
    ||\left(1 - \chi_x\right)\left(e^{r_1\Delta}P\right)^{l+1}\delta_x||
    \end{equation}
    because $e^\Delta$, $\chi_x$ and $P$ have operator norm $1$.  We start by estimating this last  quantity when $l=0$. We remark first of all that 
    \begin{equation} ||(1-\chi_x)e^{r_1\Delta}P\delta_x|| \leq ||(1-\chi_x)e^{\left(r_1-1\right)\Delta}\left|e^{\Delta}P\delta_x\right||| \label{Eq: trianle ineq}\end{equation}
    according to the triangle inequality. In order to make sense of the right-hand side, notice that $e^{\Delta}P\delta_x$ is equal to the element $\phi$ of $L^2(M,\nu)$ defined by
    $$ \phi: y \longmapsto K_1(y,x) - \psi_{k_x}(x)\int_MK_1(y,z)\psi_{k_x}(z)d\nu(z)$$
    Here $k_x$ denotes the unique $k\in \net$ such that $\psi_{k}(x) \neq 0$ if it exists (then $k_x$ is automatically unique), and $\psi_{k_x}=0$ otherwise. The notation $|e^\Delta P\delta_x|$ then corresponds to the absolute value of $\phi$. 
    
    We provide now an estimate of the right-hand side of \eqref{Eq: trianle ineq}. As a general fact, we have for any $t>0$,
    \begin{align*}
    ||(1-\chi_x)e^{t\Delta}\left|e^{\Delta}P\delta_x\right|||   \leq ||\left(1 - \chi_x\right)e^{t\Delta}K_1(\cdot,x)|| + C_0||\left(1 - \chi_x\right)e^{t\Delta}(e^\Delta\psi_{k_x})||
    \end{align*}
    where we have used that $|\psi_{k_x}(x)|\leq C_0$ for any $x$ (see Section \ref{s:voronoi}).
   % where we have used that  $$ \phi(y) \leq K_1(y,x) + \psi_{k_x}(x)\int_MK_1(y,z)\psi_{k_x}(z)dz$$  to go from the first to the second line. 
We have moreover according to \eqref{e:defcutchi}
    \begin{equation} \label{e:Atk1}
    ||\left(1 - \chi_x\right)e^{t\Delta}K_1(\cdot,x)|| \leq ||K_{t+1}(\cdot,x)||_{L^2(M\setminus B_d(x, \cutchi r_2))}.
    \end{equation}
Now, since the $\psi_k$'s are bounded above by $C_0$ and supported on subsets of diameter bounded above by $2$, using \eqref{e:heatkernelinM} and Lemma \ref{l:comparisonheat} we obtain that there is a constant $C_1$ such that $K_1(y,x)\geq k_1(d(x,y))\geq k_1^{(b)}(d(x,y)) \geq C_1 \psi_{k_x}(y)$. Hence,
\begin{align*} 
||\left(1 - \chi_x\right)e^{t\Delta}(e^\Delta\psi_{k_x})|| \leq C_1||\left(1 - \chi_x\right)e^{(t+1)\Delta} K_1(\cdot, x)||\leq  C_1||K_{t+2}(\cdot,x)||_{L^2(M\setminus B_d(x, \cutchi r_2))}
\end{align*}
According to Lemma \ref{l:heatkernelinM},
    \begin{align}
        ||K_{t+2}(\cdot,x)||_{L^2(M\setminus B_d(x, \cutchi r_2))} &\leq C_{2}\exp(2|b|t)||K_{t+2}(\cdot,x)||^{1/2}_{L^1(M\setminus B_d(x, \cutchi r_2))}\nonumber \\
        &\leq C_{2}\exp(2|b|t)||k_{t+2}(\cdot,x)||^{1/2}_{L^1(\widetilde{M}\setminus B_{\widetilde{M}}(x, \cutchi r_2))}\nonumber
    \end{align} 
where in the last line we used \eqref{e:heatkernelinM}.
The same argument applied to the right-hand side of \eqref{e:Atk1} yields the same bound, at time $t+1$. From now on, we assume $t\in[0,2r_2]$. Using the heat kernel estimate \eqref{e:L1norm} applied with $\cutheat=\cutchi r_2/(t+2)\geq 8\sqrt{|b|}+4$ and $\cutheat=\cutchi r_2/(t+1)\geq 8\sqrt{|b|}+4$ we thus have 
    \begin{equation}\label{Eq: estimate Chix with hk}
    ||(1-\chi_x)e^{t\Delta}\left|e^\Delta P\delta_x\right||| \leq C_3\exp\left(\frac52 |b|t-\frac{\cutchi^2r_2^2}{8(t+2)}\right).
    \end{equation}
Applying to $t=r_1-1$ and combining with \eqref{Eq: trianle ineq} we get an upper bound for \eqref{e:onveutunebornesup} for $l=0$.    
    
    We turn now to the case $l \geq 1$. For any $f \in L^2(M,\nu)$, $t \geq 0$ and $y \in M$ we have 
    \begin{align*}
        |Pe^{t\Delta}f(y)|
%& = | \int_M K_t(y,z)f(z)dz - \sum_k \psi_k(y) \int_M \psi_k(w)K_t(w,z)f(z)dwdz| \\
                         & \leq e^{t\Delta}|f|(y) + \sum_{k\in\net} \psi_k(y) \int_{M\times M} \psi_k(w)K_t(w,z)|f(z) |d\nu(w)d\nu(z).
    \end{align*}
    The $\psi_k$'s are bounded by $C_0$ and supported on pairwise disjoint subsets of diameter bounded above by $2$. Therefore, there is a constant $C_4 > 0$ such that for any $y,w\in M$,
    $$ \sum_{k\in\net} \psi_k(y)\psi_k(w)\leq C_4 k_1^{(b)}(d_{\widetilde{M}}(y,w)) \leq C_4 k_1(y,w) \leq C_4 K_1(y,w)$$
   (where we used Lemma \ref{l:comparisonheat} and \eqref{e:heatkernelinM}). This yields    
    \begin{align*}
        \sum_{k\in\net} \psi_k(y) \int_M \psi_k(w)K_t(w,z)|f(z) |d\nu(w)d\nu(z) 
%& \leq  C_\rho'\int_M K_1(y,w)K_t(w,z)|f(z) |dwdz \\
              & \leq C_4e^{\left(t+1\right) \Delta}|f|(y).
    \end{align*}
    Hence, 
    \begin{equation}|Pe^{t\Delta}f(y)| \leq (e^{t\Delta} + C_4 e^{\left(t+1\right)\Delta})|f|(y). \label{Eq: upper bound kernel + proj}\end{equation}
    A simple induction then shows that,
    $$ |P\left(e^{t\Delta}P\right)^{l-1}f(y)| \leq (e^{t\Delta} + C_4 e^{\left(t+1\right)\Delta})^{l-1}|Pf|(y).$$
    Applying the above with $t=r_1$ and $f=e^{r_1\Delta}P\delta_x$ yields
    \begin{align*}||\left(1 - \chi_x\right)\left(e^{r_1\Delta}P\right)^{l+1}\delta_x|| 
%&= ||\left(1 - \chi_x\right)\left(e^{r_1\Delta}P\right)^le^{r_1\Delta}P\delta_x|| \\
    &\leq ||(1- \chi_x) e^{r_1\Delta}(e^{r_1\Delta} + C_4e^{\left(r_1 + 1\right)\Delta})^{l-1}|Pe^{r_1\Delta}P\delta_x|||.
    \end{align*}
    Applying \eqref{Eq: upper bound kernel + proj} with $t=r_1 - 1$ and $f= e^\Delta P\delta_x$ we find 
    $$ |Pe^{r_1\Delta}P\delta_x|(y) \leq (e^{\left(r_1-1\right)\Delta} + C_4 e^{r_1\Delta })|e^\Delta P\delta_x|(y)$$
    for all $y \in M$.  So the triangle inequality yields
    \begin{align}
        ||\left(1 - \chi_x\right)\left(e^{r_1\Delta}P\right)^{l+1}\delta_x|| 
%& \leq ||(1- \chi_x) e^{r_1\Delta}(e^{r_1\Delta} + C_\rho'e^{\left(r_1 + 1\right)\Delta})^{l-1}(A^{r_1-1} + C_\rho' A^{r_1 })|AP\delta_x||| \\
        & \leq \sum_{j=0}^{l} \binom{l}{j} C_4^{j}||(1 - \chi_x)e^{\left((l+1)r_1 + j - 1\right)\Delta}\left|e^\Delta P\delta_x\right| ||. \label{Eq: Upper bound heat with P}
    \end{align}
    All in all, we have, using successively \eqref{Eq: Upper bound heat with P} and \eqref{Eq: estimate Chix with hk} (since $(l+1)r_1 + j + 1 \leq (l+1)(r_1+1)\leq 2r_2$) 
    \begin{align}
    ||\left(P\chi_xe^{r_1\Delta}\right)^{n-l-1}P\left(1 - \chi_x\right)\left(e^{r_1\Delta}P\right)^{l+1}\delta_x|| & \leq \sum_{j=0}^{l} \binom{l}{j} C_4^{j}||(1 - \chi_x)e^{\left((l+1)r_1 + j  - 1\right)\Delta}\left|e^\Delta P\delta_x\right| || \nonumber\\
    &\leq \sum_{j=0}^{l} \binom{l}{j}C_4^{j} C_3\exp\left(\frac52 |b|r_2-\frac{\cutchi^2r_2^2}{8\left(l+1\right)\left(r_1 + 1 \right)}\right) \nonumber\\
    & =(1+C_4)^{l}C_3 \exp\left(\frac52 |b|r_2-\frac{\cutchi^2r_2^2}{8\left(l+1\right)\left(r_1 + 1 \right)}\right). \label{Eq: Estimates diff cutoff 1}
    \end{align}
Finally, combining \eqref{Eq: telescoping sum} with the estimate \eqref{Eq: Estimates diff cutoff 1} and the estimate on the $l=0$ term, we find 
    \begin{align}
        &||\left(\left(Pe^{r_1\Delta}P\right)^n - \left(P\chi_xe^{r_1\Delta}\chi_xP\right)^n \right)\delta_x||\nonumber\\
        & \qquad \leq  C_3\exp\left( \frac52 |b|r_2-\frac{\cutchi^2r_2^2}{8(r_1+1)}\right) + \sum_{l=1}^{n-1} (1+C_4)^{l}C_3 \exp\left( \frac52 |b|r_2-\frac{\cutchi^2r_2^2}{8\left(l+1\right)\left(r_1 + 1 \right)}\right)\label{e:22} \\
        &\qquad \leq C_3\exp\left(\frac52 |b|r_2-\frac{\cutchi^2r_2^2}{8(r_1+1)}\right)+ C_3n(1+C_4)^{n} \exp\left(\frac52 |b|r_2-\frac{\cutchi^2r_2^2}{8n\left(r_1 + 1 \right)}\right)\label{e:33}\\
         &\qquad \leq C_5\exp\left(-\frac{\cutchi^2r_2}{32}\right)\label{e:44}
    \end{align}
    where we have estimated all terms of the sum by the maximal one in going from \eqref{e:22} to \eqref{e:33}, and we used that $n(r_1+1) \leq 2r_2$ and that $n(1 + C_4)^n$ is negligible compared to $\exp(|b|r_2/2)$ to go from \eqref{e:33} to \eqref{e:44} (together with the fact that $\cutchi\geq 32\sqrt{|b|}+16$). 
  \end{proof}

The end of this section is devoted to the proof of Lemma \ref{l:comparedeltatophi}.
\begin{proof}[Proof of Lemma \ref{l:comparedeltatophi}]
Notice that the operator $P\chi_x e^{r_1\Delta} \chi_x P$ is compact as an operator on $L^2(M,\nu)$, as a composition of bounded (linear) operators with the compact operator $e^{r_1\Delta}$.
Now we notice that
$$
\|(Pe^{r_1\Delta}P)^{\lfloor r_2/r_1\rfloor+1}\delta_x\|\leq  \|(P\chi_xe^{r_1\Delta}\chi_xP)^{\lfloor r_2/r_1\rfloor+1}\delta_x\|+C_0 \exp\left(-\frac{\cutchi^2 r_2}{32}\right)
$$
for some $C_0>0$ due to Lemma \ref{l:simoooon}. Then we set $B_x=P\chi_xe^{r_1\Delta}\chi_xP$. We prove below that $\|B_x\delta_x\|\leq C_1$ for some $C_1>0$ (as always, depending only on $(a,b,\rho)\in\S$). Therefore,
\begin{align}
\|(Pe^{r_1\Delta}P)^{\lfloor r_2/r_1\rfloor+1}\delta_x\|&\leq \|B_x^{\lfloor r_2/r_1\rfloor}B_x\delta_x\|+C_0 \exp\left(-\frac{\cutchi^2 r_2}{32}\right)\nonumber\\
&\leq C_1\sup_{\|\varphi\|=1}\|B_x^{\lfloor r_2/r_1\rfloor}\varphi\|+C_0\exp\left(-\frac{\cutchi^2 r_2}{32}\right),\nonumber
\end{align}
which is exactly \eqref{e:deltaxetphiavecreste}.
There remains to justify that $\|B_x\delta_x\|\leq C_1$. We first have
\begin{equation}\label{e:onyestpresque1}
\|B_x\delta_x\|\leq \|e^\Delta \chi_x P\delta_x\|
\end{equation}
since the operator norms of $P$, $\chi_x$ and $e^{(r_1-1)\Delta}$ are all equal to $1$. If we denote by $k_x$ the only $k\in\net$ such that $x \in V_{k_x}$ when it exists (and $\psi_{k_x}=0$ otherwise), then by definition of $P$ we have
\begin{equation}\label{e:APdeltax}
\|e^{\Delta}\chi_x P\delta_x\|=  \|e^{\Delta} P\chi_x\delta_x\| \leq \|e^{\Delta}\delta_x\|+ C_0\|e^{\Delta}\psi_{k_x}\|  \leq \|e^{\Delta}\delta_x\| + C_0
\end{equation}
where we used that $\|\psi_k\|=1$ and that the operator norm of $e^\Delta :L^2(M,\nu)\rightarrow L^2(M,\nu)$ is equal to $1$.
Using \eqref{e:heatkernelinM}, we get
\begin{equation}\label{e:onyestpresque2} \|e^{\Delta}\delta_x\|^2=\|K_1(x,\cdot)\|^2 \leq \|K_1(x,\cdot)\|_{L^\infty}\|K_1(x,\cdot)\|_{L^1}= \|K_1(x,\cdot)\|_{L^\infty} \leq C_2
\end{equation}
where the last inequality is a consequence of Lemma \ref{l:heatkernelinM}. Combining \eqref{e:onyestpresque1}, \eqref{e:APdeltax}, \eqref{e:onyestpresque2}, we get $\|B_x\delta_x\|\leq C_1$, which concludes the proof.
\end{proof}

\subsection{Applications of the min-max}\label{s:applicationsminmax}
Our proof of Theorem \ref{t:volumebound} relies fundamentally on the  Courant–Fischer min-max lemma (see e.g. \cite[Theorem XIII.1]{reed}) through the lemma proved in this section. In all the sequel, 
\begin{equation*}
\mu_2=e^{-\lambda_2(M)}
\end{equation*}
denotes the largest eigenvalue of $e^{\Delta}$ strictly smaller than $1$. 
%\begin{lemma}\label{l:eqoflambda1}There holds\begin{equation}\label{e:eqoflambda1}\sup_{\|\varphi\|=1}\|(P\chi_xe^{r_1\Delta}\chi_xP)^{\lfloor r_2/r_1\rfloor}\varphi\|=\sup_{\|\varphi\|=1}\|(P\chi_xe^{r_1\Delta}\chi_xP)\varphi\|^{\lfloor r_2/r_1\rfloor}\end{equation}\end{lemma}\begin{proof}The operator $P\chi_xe^{r_1\Delta}\chi_xP:L^2(M,\nu)\rightarrow L^2(M,\nu)$ is selfadjoint and compact (because $e^{\Delta}$ is compact). Hence by the min-max theorem, the left-hand side of \eqref{e:eqoflambda1} is equal to the top eigenvalue of $(P\chi_xe^{r_1\Delta}\chi_xP)^{2\lfloor r_2/r_1\rfloor}$, which is also equal to the top eigenvalue of $(P\chi_xe^{r_1\Delta}\chi_xP)^2$, raised to the power $\lfloor r_2/r_1\rfloor$. Using again the min-max theorem we obtain \eqref{e:eqoflambda1}.\end{proof}

As in Section \ref{s:errorterms}, the main quantity of interest in this section is, for any fixed $x\in M$,
\begin{equation}\label{e:defofphix}
\sup_{\|\varphi\|=1}\|P\chi_xe^{r_1\Delta}\chi_xP\varphi\|.
\end{equation}
We denote by $\varphi_x\in L^2(M,\nu)$ a function which attains the supremum \eqref{e:defofphix}, i.e., 
\begin{equation}\label{e:defvarphixxx}
\|P\chi_xe^{r_1\Delta}\chi_xP\varphi_x\|=\sup_{\|\varphi\|=1}\|P\chi_xe^{r_1\Delta}\chi_xP\varphi\|, \qquad \|\varphi_x\|=1.
\end{equation}
By the min-max principle $\varphi_x$ is an eigenfunction of the compact, self-adjoint and non-negative operator $P\chi_xe^{r_1\Delta}\chi_xP$. Since $[\chi_x,P]=0$ and $P$ and the operator of multiplication by $\chi_x$ are orthogonal projections,  the eigenfunction $\varphi_x$ verifies \begin{equation}\label{e:eqsatisfiedbyphix}
P\varphi_x=\varphi_x \quad \text{and} \quad \chi_x\varphi_x=\varphi_x.
\end{equation} 

The following result serves as a replacement for the bound on the set $U$ in the proof of \cite[Theorem 2.2]{equiangular}. 
\begin{lemma}\label{l:defofS}
Assume that $\frac{\cutchi^2}{16}\geq 5|b|+2\cutchi \sqrt{|b|}+\cutchi$. Then there exists $C_0,C_2>0$ and a subset $S\subset M$ of area $\nu(S)\leq C_0 \exp(3\cutchi r_2\sqrt{|b|})$ such that for any $x\notin S$, 
\begin{equation}\label{e:argU0}
\|e^{r_1\Delta}|\varphi_x|\|^2\leq \mu_2^{2r_1}+C_2\exp(-\cutchi r_2).
\end{equation} 
\end{lemma}
\begin{proof}
Assume by contradiction that it is possible to find $x_1, x_2\in M$ at distance $> 3 \cutchi r_2+4$ such that \eqref{e:argU0} does not hold, with $C_2=2C_0C_1$, where $C_0,C_1$ are the constants appearing respectively in  \eqref{e:L1norm} and Lemma \ref{l:heatkernelinM}. Then $\|e^{r_1\Delta}|\varphi_{x_1}|\|^2$ and $\|e^{r_1\Delta}|\varphi_{x_2}|\|^2$ are both strictly greater than $\mu_2^{2r_1}+C_2\exp(-\cutchi r_2)$. We show below that 
\begin{equation}\label{e:doctorant}
|\langle e^{2r_1\Delta}|\varphi_{x_1}|,|\varphi_{x_2}|\rangle|\leq C_0C_1  \exp(-\cutchi r_2).
\end{equation}
Before proving \eqref{e:doctorant}, let us explain how to finish the proof. Let $u,v\geq 0$ with $u^2+v^2=1$ such that $\varphi=u|\varphi_{x_1}|+v|\varphi_{x_2}|$ verifies $\int_M \varphi(x) d\nu(x)=0$. Then $\|\varphi\|=1$ since $\varphi_{x_1}$ and $\varphi_{x_2}$ have disjoint support (due to \eqref{e:eqsatisfiedbyphix}) and $\|e^{r_1\Delta}\varphi\|>\mu_2^{r_1}$, so $\varphi$ contradicts the min-max principle for $e^{2r_1\Delta}$ in $L^2(M,\nu)$. This proves \eqref{e:argU0} and shows that $S$ has diameter $\leq 3\cutchi r_2+4$.
Finally, take $x\in S$ (if not empty) and notice that $$\nu(S)\leq \nu(B_d(x,3\cutchi r_2+4))\leq \frac{4\pi}{|b|} \exp((3\cutchi r_2+4)\sqrt{|b|})$$
by \eqref{e:volumeballmodel}, which concludes the proof.\\
There remains to prove \eqref{e:doctorant}. For $y\in M$ such that $\chi_{x_2}(y)=1$, we have according to \eqref{e:L1norm} and Lemma \ref{l:heatkernelinM}
\begin{align*}
\int_M \chi_{x_1}(x) K_{2r_1}(x,y)^2d\nu(x)&\leq C_{1}\exp(4|b|t)\int_M \chi_{x_1}(x) K_{2r_1}(x,y)d\nu(x)\leq C_0C_{1} \exp\Bigl(5|b|r_2-\frac{\cutchi^2r_2}{16}\Bigr)
\end{align*}
since the supports of $\chi_{x_1}$ and $\chi_{x_2}$ are at distance $\geq \cutchi r_2$ (recall \eqref{e:defcutchi} and the fact that the $V_k$ have diameter at most $2$).
We deduce using Cauchy-Schwarz and $\|\varphi_{x_1}\|=\|\varphi_{x_2}\|=1$ that  
\begin{align*}
\langle e^{2r_1\Delta}|\varphi_{x_1}|,|\varphi_{x_2}|\rangle^2&\leq \int_{M\times M}\chi_{x_1}(x)\chi_{x_2}(y)K_{2r_1}(x,y)^2d\nu(x)d\nu(y)\\
&\leq C_0C_{1} \exp\Bigl(\sqrt{|b|}(C'r_2+1)+5|b|r_2-\frac{\cutchi^2r_2}{16}\Bigr)
\end{align*}
which implies \eqref{e:doctorant} if $\cutchi$ is chosen as in the statement.
\end{proof}

\subsection{Estimate of the gain $\varepsilon$}\label{s:gainestimates}
The heart of our proof, the ``gain of $\varepsilon$", corresponding to Step (v) in the strategy of proof in Section \ref{s:strategy}, is carried out in this section. 

At the beginning of the proof of Theorem \ref{t:volumebound} in Section \ref{s:proofofth}, we will prove using Section \ref{s:applicationsminmax} that $m'$ is controlled by the integral in $x$ of the quantity $\|P\chi_x e^{r_1\Delta}\chi_x P\varphi_x\|$ where $\varphi_x$ has been introduced in \eqref{e:defvarphixxx}.

In the present section we show that the quantity $\|P\chi_x e^{r_1\Delta}\chi_x P\varphi_x\|=\|P\chi_x e^{r_1\Delta}\varphi_x\|$ (see \eqref{e:eqsatisfiedbyphix}) is, in turn, bounded above by $(1-\varepsilon)\|e^{r_1\Delta}|\varphi_x|\|$ up to remainders, where $\varepsilon$ is ``not too small" (depending on the volume). This gain of a quantitative $\varepsilon$, combined with an upper bound on $\|e^{r_1\Delta}|\varphi_x|\|$ obtained through Lemma \ref{l:defofS}, is sufficient to conclude the proof of Theorem \ref{t:volumebound} in Section \ref{s:proofofth}.

The gain of $\varepsilon$ is proved by writing the identity (full details are provided in the proof of Lemma \ref{l:gain})
\begin{align*}
\|e^{r_1\Delta}|\varphi_x|\|^2-\|P\chi_x e^{r_1\Delta}\chi_xP\varphi_x\|^2& =\underbrace{4\langle e^{r_1\Delta}\varphi_+,e^{r_1\Delta}\varphi_-\rangle}_{:=G_1}+\underbrace{\sum_{k\in\net_x}(\langle e^{r_1\Delta}\varphi_+,\psi_k\rangle-\langle e^{r_1\Delta}\varphi_-,\psi_k\rangle)^2}_{:=G_2}\\
&\qquad+\|(1-\chi_x)e^{r_1\Delta}\varphi_x\|^2
\end{align*}
where $\varphi_\pm=\max(\pm\varphi_x,0)$ and $\net_x$ is introduced in \eqref{e:netx}. The key observation is that $G_1$ and $G_2$ play opposite roles. While $G_1$ quantifies the interaction (or lack thereof) between $e^{r_1\Delta}\varphi_+$ and $e^{r_1\Delta}\varphi_-$, $G_2$ measures the discrepancy between the mass left by $e^{r_1\Delta}\varphi_+$ and $e^{r_1\Delta}\varphi_-$ on the Voronoi cells $V_k$. Since the variations of $e^{r_1\Delta}\varphi_+$ and $e^{r_1\Delta}\varphi_-$ are very well controlled at scale $O(1)$ (see Lemma \ref{l:smallscaleinvar}, which is a consequence of \eqref{e:variationsalpha}), $e^{r_1\Delta}\varphi_+$ and $e^{r_1\Delta}\varphi_-$ must interact on the Voronoi cells $V_k$ in order for $G_2$ to be small. In turn, this prevents $G_1$ from being small.  In other words, $G_1 + G_2$ (hence $\varepsilon$) cannot be small. This will be expressed in practice as a lower bound for $G_1 + G_2$ in terms of the $L^2$-norm of $e^{r_1\Delta}|\varphi_x|$.

The following lemma, used in the end for $r=r_1$ but valid for any $r\geq 1$, illustrates the idea that the solutions of the heat equation at time $r$ do not vary too much over scales of size $\lesssim r$. We set 
\begin{equation}\label{e:w1w2}
w_{\cutheat}(r)=C_0\exp\left(-5(1+|b|)(1+\cutheat)r\right).
\end{equation}
where $C_0$ is given in \eqref{e:variationsalpha}. 
\begin{lemma}[Small scale invariance] \label{l:smallscaleinvar}
There exists $C_0 > 0$ such that for any $\cutheat \geq \max(8\sqrt{|b|}+4,16)$, any $r\geq 1$, and any  positive function $f$ with $\|f\|=1$ there exists $R\in C^0(M)$ with $$\|R\| \leq C_0 \exp\left(|b|r-\frac{\cutheat^2r}{64}\right)$$ and the inequality 
 $$
 e^{r\Delta}f(x)\geq w_{\cutheat}(r)(e^{r\Delta}f(x') -R(x'))\geq 0
 $$
holds for any $x,x'\in M$ at distance at most $4r+4$ from each other. 
\end{lemma}
\begin{proof}
Recall that $e^{r\Delta}f(x)=\int_MK_{r}(x,y)f(y)dy$. For every $x,y \in M$ define $$K_{r}^{\leq \cutheat}(x,y):=\sum_{\gamma \in \Gamma,\ d_{\hyp}(\bar{x},\gamma \bar{y}) \leq \cutheat r} k_{r}(\bar{x},\gamma \bar{y})$$
and 
$$K_{r}^{> \cutheat}(x,y):=\sum_{\gamma \in \Gamma,\ d_{\hyp}(\bar{x},\gamma \bar{y}) > \cutheat r} k_{r}(\bar{x},\gamma \bar{y})$$
for any choice of lifts $\bar{x},\bar{y}$ of $x,y$ to $\hyp$. We have $K_{r}=K_{r}^{>C}+K_{r}^{\leq C}$ according to \eqref{e:heatkernelinM}.

For $x\in M$ we set
\begin{equation}\label{e:RC}
R_{>\cutheat}(x):= \int_{M} K_{r}^{>\cutheat}(x,y)f(y)d\nu(y).
\end{equation}
Notice that $R_{>\cutheat} \in C^{0}(M)$.
It follows from the Schur test that
\begin{equation}\label{e:rckc}
||R_{>\cutheat}||_{L^2} \leq \sup_{x \in M} ||K_{r}^{>\cutheat}(x, \cdot)||_{L^1}||f||_{L^2}.
\end{equation} 

Take $x \in M$ and choose $\bar{x} \in \hyp$ a lift of $x$. We have 
\begin{align}
||K_{r}^{>\cutheat}(x, \cdot)||_{L^1} 
%& = \int_M K_{r_1}^{>\cutheat}(x, y)dy \nonumber\\
%& = \int_\Omega \sum_{\gamma \in \Gamma, d_{\hyp}(\bar{x},\gamma \bar{y}) > \cutheat r_1} k_{r_1}(\bar{x},\gamma \bar{y}) d\bar{y}\nonumber\\
%& = \sum_{\gamma \in \Gamma} \int_{\gamma\Omega \setminus B_{\hyp}(\bar{x}, \cutheat r_1)} k_{r_1}(\bar{x},\bar{y}) d\bar{y} \nonumber\\
= \int_{\hyp\setminus B_{\hyp}(\bar{x},\cutheat r)}k_{r}(\bar{x},\bar{y})d\bar{y}  \leq C_0 \exp\left(|b|r-\frac{\cutheat^2r}{16}\right)\label{e:c5}
\end{align}
according to \eqref{e:L1norm}. Besides, $K^{\leq \cutheat}$ has controlled variations in terms of $\cutheat$: we prove that for all $x, x', y \in M$ with $d(x,x') \leq 4r+4$ we have
\begin{equation}\label{e:estdeKr1}
K_{r}^{\leq \cutheat}(x,y) \geq w_{\cutheat}(r)K_{r}^{\leq \cutheat - 8}(x',y).
\end{equation}
%\leq C_0 w_2(r)K_{r_1}^{\leq C + 1}(x',y)
For this, we fix lifts $\bar{x}$, $\bar{x}'$ of $x,x'$ at distance $\leq 4r+4$. Using that if $d_{\hyp}(\bar{x}',\gamma \bar{y})\leq (\cutheat-8)r$, then $d_{\hyp}(\bar{x},\gamma \bar{y})\leq  \cutheat r$, together with \eqref{e:variationsalpha} applied with $t=r$ we get
\begin{align*}
K_{r}^{\leq \cutheat}(x,y)&=\sum_{\gamma \in \Gamma, d_{\hyp}(\bar{x},\gamma \bar{y}) \leq \cutheat r} k_{r}(d_{\hyp}(\bar{x},\gamma \bar{y}))\\
&\geq w_{\cutheat}(r)\sum_{\gamma \in \Gamma, d_{\hyp}(\bar{x},\gamma \bar{y}) \leq \cutheat r} k_{r}(d_{\hyp}(\bar{x}',\gamma \bar{y}))\\
&\geq w_{\cutheat}(r)\sum_{\gamma \in \Gamma, d_{\hyp}(\bar{x}',\gamma \bar{y}) \leq (\cutheat-8)r} k_{r}(d_{\hyp}(\bar{x}',\gamma \bar{y}))= w_{\cutheat}(r)K_{r}^{\leq \cutheat-8}(x',y).
\end{align*}
We deduce from \eqref{e:estdeKr1} and positivity of the heat kernel
$$ e^{r\Delta}f(x) \geq w_{\cutheat}(r) (e^{r\Delta}f(x') -R_{>\cutheat-8}(x')).$$
Set $R=\min(R_{>\cutheat-8},e^{r\Delta}f)$. Combining \eqref{e:rckc}, \eqref{e:c5}  and the inequality $\cutheat-8 \geq \cutheat/2$ we obtain the lemma.
\end{proof}

For $f\in L^2(M,\nu)$ we set $f_\pm=\max(\pm f,0)$. The next lemma, when used for $r=r_1$, shows that the interaction $\langle e^{r_1\Delta}f_+,e^{r_1\Delta}f_-\rangle$ between positive and negative parts can already be detected coarsely on the Voronoi cells $V_k$. Its proof relies on Lemma \ref{l:smallscaleinvar}.

\begin{lemma}\label{e:lesmk}
There exist $C_0,C_1>0$ such that for any $\cutheat \geq \max(8\sqrt{|b|}+4,16)$, any $r\geq 1$, and any $f\in L^2(M,\nu)$ with $\|f\|=1$, there holds 
$$
\langle e^{r\Delta}f_+,e^{r\Delta}f_-\rangle\geq C_0 w_{\cutheat}(r)^2\sum_{k\in\net}m_k^+m_k^-- C_1 \exp\left(|b|r-\frac{\cutheat^2r}{64}\right)
$$
where $m_k^\pm=\langle f_\pm,e^{r\Delta}\psi_k\rangle$.
\end{lemma}
Instead of $w_{\cutheat}(r)^2$ we could put the better $w_{\cutheat}(2)^2$ in the right-hand side. But for simplicity of notation, since $w_{\cutheat}(r)$ appears at other places in the sequel whereas $w_{\cutheat}(2)$ does not, we present the statement in the above form.
\begin{proof}[Proof of Lemma \ref{e:lesmk}]
We denote by $R^{\pm}$ the remainder corresponding to $f^\pm$ in Lemma \ref{l:smallscaleinvar}. We recall that $C_0\leq \vol(\supp(\psi_k))\leq C_1$ (see \eqref{e:volmin}).

Let $k\in\net$ and $x\in \supp(\psi_k)$. Averaging over $x'$ in the support of $\psi_k$ and using that $4r+4\geq 2\geq \text{diam}(V_k)$ and $\|\psi_k\|_{L^\infty}\leq C_2$ (see Section \ref{s:voronoi}), we deduce from Lemma \ref{l:smallscaleinvar} that
\begin{align}
e^{r\Delta}f_\pm(x) \geq \frac{ C_2^{-1}w_{\cutheat}(r)}{\vol(\supp(\psi_k))}(\langle e^{r\Delta}f_\pm,\psi_k\rangle -\langle R^\pm,\psi_k\rangle)\geq C_3w_{\cutheat}(r)m_k^{\pm} -C_4\|R^\pm\|_{L^2(\supp(\psi_k))}\geq 0\label{e:Ar1f+}
\end{align} 
where we used the Cauchy-Schwarz inequality for the second term. 
%The remainder $$R_\pm(x)=\begin{cases}\|R^\pm\|_{L^2(\supp(\psi_{k_x}))} &\text{if } x\in\supp(\psi_{k_x}) \text{ for some } k_x\in\net \\0 &\text{otherwise}\end{cases}$$ verifies $\|R_\pm\|\leq C_1\|R^\pm\|_{L^2(\bigcup_{k\in\net}\supp(\psi_k))}$. 

We notice that 
\begin{align*}
\sum_{k\in\net} m_k^-\|R^+\|_{L^2(\supp(\psi_{k}))}&\leq \left(\sum_{k\in\net} \|R^+\|_{L^2(\supp(\psi_{k}))}^2\right)^{\frac12} \left(\sum_{k\in\net} \left(m_k^-\right)^2\right)^{\frac12}\\
&\leq C_5\exp\left(|b|r-\frac{\cutheat^2r}{64}\right)
\end{align*}
and similarly for the sum with switched signs, i.e., with $m_k^+$ and $R^-$. Then using \eqref{e:Ar1f+} and the fact that $w_{\cutheat}(r)\leq 1$ we obtain
\begin{align*}
\langle e^{r\Delta}f_+,e^{r\Delta}f_-\rangle&\geq  \sum_{k\in\net} \int_{\supp(\psi_k)} e^{r\Delta}f_+(x)e^{r\Delta}f_-(x)d\nu(x)\\
&\geq C_0\sum_{k\in\net} \left(C_3w_{\cutheat}(r) m_k^+ -C_4 \|R^+\|_{L^2(\supp(\psi_{k}))}\right)\left(C_3w_{\cutheat}(r) m_k^- -C_4 \|R^-\|_{L^2(\supp(\psi_{k}))}\right)\\
%&\geq \sum_{k\in\net}\vol(\supp(\psi_k))C_1^2\alpha^2m_k^+m_k^- - C_2(\|R_+(x)\|+\|R_-(x)\|)\\
&\geq C_6w_{\cutheat}(r)^2\sum_{k\in\net}m_k^+m_k^-- C_7 \exp\left(|b|r-\frac{\cutheat^2r}{64}\right)
\end{align*}
which concludes the proof.
\end{proof}

We provide now what will serve as the lower bound for $G_1+G_2$ alluded to at the beginning of this section. 

\begin{lemma}\label{e:Ar1|f|phik}
For any $x \in M$ set 
\begin{equation}\label{e:netx}
    \net_x=\{k\in\net\mid \supp(\psi_k) \cap B_d(x,\cutchi r_2) \neq \emptyset\}.
\end{equation}
There exist constants $C_0,C_1>0$ such that for any $\cutheat \geq \max(8\sqrt{|b|}+4,16)$,  any $f\in L^2(M,\nu)$ with $\|f\|=1$, any $x\in M$, 
\begin{align*}
\sum_{k\in\net_x} \langle e^{r_1\Delta}|f|,\psi_k\rangle^2&\geq C_0\exp(-2r_1\sqrt{|b|})w_{\cutheat}(r_1)^2\|\chi_x e^{r_1\Delta}|f|\|^2 -C_1 \exp\left(|b|r_1-\frac{\cutheat^2r_1}{64}\right).
\end{align*}
\end{lemma}
\begin{proof}
For any $k\in\net$, we denote by $y_k$ a point where $e^{r_1\Delta}|f|$ attains its minimum on $V_k=\supp(\psi_k)$. We first show that 
\begin{equation}
\label{e:inclusiondusupportdechi}
\supp(\chi_x)\subset \bigcup_{k\in \net_x}B_d(y_k,2r_1+2).
\end{equation} 
Let $z\in\supp(\chi_x)$. If $z\in\supp(\chi_x)\setminus B_d(x,\cutchi r_2)$, then $z\in V_k$ for some $k\in\net$, and $d(z,y_k)\leq \text{diam}(V_k)\leq 2r_1+2$. Otherwise, $z\in B_d(x,\cutchi r_2)$, and therefore there exists $z'\in B_d(x,\cutchi r_2-r_1)$ at distance at most $r_1$ from $z$. We have $B_d(z’, r_1)\subset B(x, \cutchi r_2)$, therefore there exists an element of the $r_1$-net $x_k\in B(x, \cutchi r_2)$ with $d(x_k,z')\leq r_1$. It follows that 
$$
d(y_k,z)\leq d(y_k,x_k)+d(x_k,z')+d(z',z)< 2r_1+2.
$$

Using Lemma \ref{l:smallscaleinvar} and noticing that $4r_1 +4$ is the diameter of the ball $B_d(y_k,2r_1+2)$, we have
\begin{align*}
\langle e^{r_1\Delta}|f|,\psi_k\rangle^2& \geq C_0(e^{r_1\Delta}|f|(y_k))^2\\
&\geq C_0\exp(-2r_1\sqrt{|b|})w_{\cutheat}(r_1)^2\int_{B_d(y_k,2r_1+2)}(e^{r_1\Delta}|f|(x')-R(x'))^2\nu(dx')
\end{align*}
where we used that $\vol(B_d(y_k,2r_1+2))\leq C_1\exp(2r_1\sqrt{|b|})$ according to \eqref{e:volumeballmodel}.
Summing over $k \in \net_x$ and using  \eqref{e:inclusiondusupportdechi}, we obtain
\begin{align*}
\sum_{k\in\net_x}\langle e^{r_1\Delta}|f|,\psi_k\rangle^2& \geq C_0\exp(-2r_1\sqrt{|b|})w_{\cutheat}(r_1)^2\int_{M}\chi_x(x')(e^{r_1\Delta}|f|(x')-R(x'))^2\nu(dx') \\ 
&\geq C_0\exp(-2r_1\sqrt{|b|})w_{\cutheat}(r_1)^2\|\chi_x e^{r_1\Delta}|f|\|^2-C_2\|R\|
\end{align*}
where in the last line we  developed the square in the right-hand side, we used the Cauchy-Schwarz inequality in $L^2(M,\nu)$ and the bound $\|e^{r_1\Delta}|f|\|\leq 1$.
\end{proof}

\begin{lemma} \label{l:gain}
Let  $\varphi_x$ be a function which attains the supremum in \eqref{e:defofphix}. There exist $C_0,C_1>0$ such that for any $\cutheat \geq \max(8\sqrt{|b|}+4,16)$ and any $x\in M$,  
\begin{align*}
\|(P\chi_xe^{r_1\Delta}\chi_xP)\varphi_x\|^2&\leq (1-C_0 \exp(-2r_1\sqrt{|b|})w_{\cutheat}(r_1)^4)\|e^{r_1\Delta}|\varphi_x|\|^2\\
&\qquad \qquad+C_1 \exp\left(|b|r_1-\frac{\cutheat^2r_1}{64}\right).
\end{align*}
\end{lemma}
\begin{proof}
We fix $x\in M$. We compute $\varepsilon=\|e^{r_1\Delta}\chi_x|\varphi_x|\|^2-\|P\chi_x e^{r_1\Delta}\chi_x P\varphi_x\|^2$ which can be simplified to
\begin{equation*}
\varepsilon=\|e^{r_1\Delta}|\varphi_x|\|^2-\|P\chi_x e^{r_1\Delta}\varphi_x\|^2
\end{equation*}
due to \eqref{e:eqsatisfiedbyphix}.
First we compute without the absolute value on $\varphi_x$, and we use from line 1 to line 2 that $\chi_x\in\{0,1\}$:
\begin{align*}
\|e^{r_1\Delta}\varphi_x\|^2-\|P\chi_x e^{r_1\Delta}\varphi_x\|^2
%&=\langle A^{2r_1}\varphi,\varphi\rangle-\langle e^{r_1\Delta}\chi_x P\chi_x e^{r_1\Delta}\varphi,\varphi\rangle\\
&=\langle (\text{Id}-\chi_x P\chi_x)e^{r_1\Delta}\varphi_x,e^{r_1\Delta}\varphi_x\rangle \\
%&=\langle (\text{Id}-\chi_x^2)e^{r_1\Delta}\varphi,e^{r_1\Delta}\varphi\rangle+\langle (\chi_x^2-\chi_x P\chi_x)e^{r_1\Delta}\varphi,e^{r_1\Delta}\varphi\rangle\\
&=\|(1-\chi_x) e^{r_1\Delta}\varphi_x\|^2+\|(\text{Id}-P)\chi_x e^{r_1\Delta}\varphi_x\|^2\\
%&=\|(1-\chi_x) e^{r_1\Delta}\varphi\|^2+\sum_{k\in\net}\langle \chi_x e^{r_1\Delta}\varphi,\psi_k\rangle^2\\
&=\|(1-\chi_x) e^{r_1\Delta}\varphi_x\|^2+\sum_{k\in\net_x}\langle e^{r_1\Delta}\varphi_x,\psi_k\rangle^2
\end{align*}
where $\net_x$ has been introduced in \eqref{e:netx}.

Then we notice the following identity:
$$
\|e^{r_1\Delta}|\varphi_x|\|^2-\|e^{r_1\Delta}\varphi_x\|^2=4\langle e^{r_1\Delta}\varphi_+,e^{r_1\Delta}\varphi_-\rangle
$$
where $\varphi_\pm = \max(\pm \varphi_x,0)$. All in all,
$$
\varepsilon=4\langle e^{r_1\Delta}\varphi_+,e^{r_1\Delta}\varphi_-\rangle+\|(1-\chi_x)e^{r_1\Delta}\varphi_x\|^2+\sum_{k\in\net_x}(\langle e^{r_1\Delta}\varphi_+,\psi_k\rangle-\langle e^{r_1\Delta}\varphi_-,\psi_k\rangle)^2
$$
(we write the last term as a difference on purpose). 
Using Lemma \ref{e:lesmk}, its notation and the fact that $m_k^++m_k^-=\langle e^{r_1\Delta}|\varphi_x|,\psi_k\rangle$ we have
\begin{align*}
\varepsilon &\geq \sum_{k\in\net_x} (m_k^+-m_k^-)^2+2\left(C_0w_{\cutheat}(r_1)^2\sum_{k\in\net} m_k^+m_k^-- C_1 \exp\left(|b|r_1-\frac{\cutheat^2r_1}{64}\right)\right) \\
& \qquad \qquad  + 2\langle e^{r_1\Delta}\varphi_+,e^{r_1\Delta}\varphi_-\rangle + \|(1-\chi_x) e^{r_1\Delta}\varphi_x\|^2\\
&\geq \min\left(1,\frac{C_0w_{\cutheat}(r_1)^2}{2}\right)\sum_{k\in\net_x} \langle e^{r_1\Delta}|\varphi_x|,\psi_k\rangle^2-2C_1 \exp\left(|b|r_1-\frac{\cutheat^2r_1}{64}\right)   + \frac{\|(1-\chi_x) e^{r_1\Delta}|\varphi_x|\|^2}{2}.
\end{align*}
Using Lemma \ref{e:Ar1|f|phik} and its notation we obtain that there exist $C_2,C_3>0$ such that
\begin{align*}
\varepsilon&\geq C_2\exp(-2r_1\sqrt{|b|})w_{\cutheat}(r_1)^4\|\chi_xe^{r_1\Delta}|\varphi_x|\|^2+\frac{\|(1-\chi_x) e^{r_1\Delta}|\varphi_x|\|^2}{2}  -C_3 \exp\left(|b|r_1-\frac{\cutheat^2r_1}{64}\right)\\
&\geq C_2\exp(-2r_1\sqrt{|b|})w_{\cutheat}(r_1)^4 \|e^{r_1\Delta}|\varphi_x|\|^2-C_3 \exp\left(|b|r_1-\frac{\cutheat^2r_1}{64}\right),
\end{align*}
which concludes the proof.
\end{proof}

\section{Proof of the main results}

\subsection{Proof of Theorems \ref{t:sublinearA} and \ref{t:volumebound}}
\label{s:proofofth}

Building upon the results of Section \ref{s:key}, we proceed with the proof of Theorem \ref{t:volumebound}, and we explain at the end of the section how to deduce Theorem \ref{t:sublinearA}. We fix $\rho>0$, $b<0$ and $M$ a closed, connected Riemannian surface with ${\rm inj}(M)\geq \rho$ and $c(x)\geq b$ for any $x\in M$. We also assume ${\rm vol}(M)\geq 3$ in order for $r_1,r_2$ to be well-defined. The case of small volumes is handled at the end.

We recall that $\mu_2=e^{-\lambda_2(M)}$. We denote by $m$ the multiplicity of $\lambda_2(M)$ as an eigenvalue of $-\Delta$. Then, $m$ is also the multiplicity of $\mu_2^{r_1}$ as an eigenvalue of $e^{r_1\Delta}$. We denote by $m'$ the multiplicity of $\mu_2^{r_1}$ as an eigenvalue of $Pe^{r_1\Delta}P$, which is also compact, self-adjoint and non-negative. 

In the sequel, we provide an upper bound on $m'$ which will be seen to be sufficient to bound $m$ (see \eqref{e:cauchyresult} below).

Since $e^{\Delta}$ is a trace-class operator, $(Pe^{r_1\Delta}P)^{2\lfloor r_2/r_1\rfloor+2}$ is also trace-class. We have by Lemma \ref{l:tracedelta}
\begin{equation*}
    m'\mu_2^{2r_1(\lfloor r_2/r_1\rfloor+1)}\leq \text{Tr}((Pe^{r_1\Delta}P)^{2\lfloor r_2/r_1\rfloor+2})=\int_M \|(Pe^{r_1\Delta}P)^{\lfloor r_2/r_1\rfloor+1}\delta_x\|^2d\nu(x)
\end{equation*}
The right-hand side is bounded above by
\begin{equation*}
C_0 \left(\exp\left(-\frac{\cutchi^2 r_2}{32}\right) {\rm vol}(M)+\int_M\sup_{\|\varphi_x\|=1}\|(P\chi_xe^{r_1\Delta}\chi_xP)^{\lfloor r_2/r_1\rfloor}\varphi_x\|^2d\nu(x) \right)
\end{equation*}
due to Lemma \ref{l:comparedeltatophi}. This last expression is equal to
\begin{equation}
C_0 \left(\exp\left(-\frac{\cutchi^2 r_2}{32}\right) {\rm vol}(M)+\int_M\left(\sup_{\|\varphi_x\|=1}\|(P\chi_xe^{r_1\Delta}\chi_xP)\varphi_x\|^2\right)^{\lfloor r_2/r_1\rfloor}d\nu(x)\right)\label{e:last}
\end{equation}
since $P\chi_xe^{r_1\Delta}\chi_xP$ is self-adjoint on $L^2(M,\nu)$.

To continue, we need to fix the parameters $\cutheat$, $\cutchi$ and the parameter $c$ introduced in \eqref{e:defr1r2}. We denote by $C_2\in (0,1)$ a constant such that $\mu_2\geq C_2$ for any $M$ of curvature $\geq b$ (thanks to Lemma \ref{l:encadrement}). We choose successively (in this order) $\cutchi$, $\cutheat$ and $c>0$ such that 
\begin{align}
\frac{\cutchi^2}{32} &\geq \max(5|b|+2\cutchi \sqrt{|b|}+\cutchi, 32\sqrt{|b|}+16, -4 \log C_2 + 1) \label{e:choicecutchi}\\
\frac{\cutheat^2}{64}&\geq |b|+25(1+|b|)(1+\cutheat)-2\log(C_2) \label{e:choicecutheat}\\
\text{and} \quad \cutheat &\geq \max(8\sqrt{|b|}+4,16) \label{e:choicecutheat22}\\
\frac{1}{4}&\geq \left(3\cutchi \sqrt{|b|}- 4\log C_2+25(1+|b|)(1+\cutheat)\right) c. \label{e:choicec}
\end{align}
We also assume ${\rm vol}(M)$ large enough so that $r_1\geq 1$.
We separate the integral in \eqref{e:last} into an integral over $S$ and an integral over $M\setminus S$, where $S$ is chosen as in Lemma \ref{l:defofS}. Due to Lemma \ref{l:defofS}, the integral over $S$ is bounded above by $\nu(S)\leq C_1\exp(3\cutchi r_2\sqrt{|b|})$ since the operator norms of $e^{r_1\Delta}$, $\chi_x$ and $P$ are equal to $1$. The integral over $M\setminus S$ is bounded above by
\begin{equation}\label{e:bigintegral}
\int_{M\setminus S} \left(\left(1-C_4\exp(-hr_1)\right)^2\|e^{r_1\Delta}|\varphi_x|\|^2+C_3\exp\left(|b|r_1-\frac{\cutheat^2r_1}{64}\right)\right)^{\lfloor r_2/r_1\rfloor}d\nu(x)
\end{equation}
by Lemma \ref{l:gain} (which we can apply thanks to \eqref{e:choicecutheat22}), where 
\begin{equation*}
h= 25(1+|b|)(1+\cutheat).
\end{equation*}
For any $x\in M\setminus S$, we have by definition of $S$
\begin{equation}\label{e:allinall}
\left(1-C_4\exp(-hr_1)\right)^2\|e^{r_1\Delta}|\varphi_x|\|^2 \leq \left(1-C_4\exp(-hr_1)\right)^2\left(\mu_2^{2r_1}+C_5\exp(-\cutchi r_2) \right)\leq  (1-C_6\exp(-hr_1))^2\mu_2^{2r_1}.
\end{equation}
Thanks to \eqref{e:choicecutheat}, 
 \begin{equation*}
\frac{\cutheat^2}{64}r_1\geq |b|r_1+hr_1-2\log(C_2)r_1.
\end{equation*}
Therefore we get, again for ${\rm vol}(M)$ large enough,
\begin{align*}
\frac{C_6}{2}\exp\left(|b|r_1-\frac{\cutheat^2r_1}{64}\right) &\leq \frac{C_6}{2}\exp(-hr_1)\mu_2^{2r_1}\\
&\leq \mu_2^{2r_1}\left(1-\frac{C_6}{2}\exp(-hr_1)\right)^2-\mu_2^{2r_1}(1-C_6\exp(-hr_1))^2.
\end{align*}
Combining with \eqref{e:allinall}, we obtain that \eqref{e:bigintegral} is bounded above by
\begin{align*}
C_0\mu_2^{2r_1\lfloor r_2/r_1\rfloor}\int_{M\setminus S}\left(1-\frac{C_6}{2}\exp(-hr_1)\right)^{2\lfloor r_2/r_1\rfloor}d\nu(x) \leq C_0 {\rm vol}(M)\left(\mu_2^{r_1}\left(1-\frac{C_6}{2}\exp(-hr_1)\right)\right)^{2\lfloor r_2/r_1\rfloor}
\end{align*}
due to Lemma \ref{l:defofS}.  Summarizing, we have obtained
\begin{align}\label{e:summarym'}
m'\mu_2^{2r_1(\lfloor r_2/r_1\rfloor+1)}&\leq C_0 {\rm vol}(M)\left(\mu_2^{r_1}\left(1-\frac{C_6}{2}\exp(-hr_1)\right)\right)^{2\lfloor r_2/r_1\rfloor}\\
&\qquad+C_0\exp(2\cutchi r_2\sqrt{|b|})+C_0\exp\left(-\frac{\cutchi^2 r_2}{32}\right){\rm vol}(M)\nonumber
\end{align}

We divide by $\mu_2^{2r_1(\lfloor r_2/r_1\rfloor+1)}$ and use the inequality $1-x\leq e^{-x}$ to deduce that $m'$ is bounded above by
\begin{equation}\label{e:grosC0foismachinbidulle}
C_0 \left(\frac{{\rm vol}(M)}{\mu_2^{2r_1}}\exp\left(-C_7\exp(-hr_1) \lfloor r_2/r_1\rfloor\right)+\frac{\exp(3\cutchi r_2\sqrt{|b|})+ \exp(-\cutchi^2 r_2/32){\rm vol}(M)}{\mu_2^{4r_2}}\right)
\end{equation}
Thanks to our choice of parameters \eqref{e:choicecutchi} and \eqref{e:choicec} we get that 
$$ \frac{\exp(3\cutchi r_2\sqrt{|b|})+ \exp(-\cutchi^2 r_2/32){\rm vol}(M)}{\mu_2^{4r_2}} \leq {\rm vol}(M)^{\frac12}+{\rm vol}(M)^{1-c}.$$ 
Thanks to \eqref{e:choicec} we have for ${\rm vol}(M)$ sufficiently large
$$ \frac{{\rm vol}(M)}{\mu_2^{2r_1}}\exp\left(-C_7\exp(-hr_1) \lfloor r_2/r_1\rfloor\right) \leq {\rm vol}(M) \exp\left( -C_8 \log^{2/3} {\rm vol}(M)\right).$$
 All in all, we find that for ${\rm vol}(M)$ sufficiently large,
\begin{equation}\label{e:mm'}
m'\leq C_0 {\rm vol}(M) \exp\left( -C_8 \log^{2/3} {\rm vol}(M)\right).
\end{equation}

%\begin{remark}\label{r:improvem'}One can in fact obtain $m'\leq C_\rho g \exp\left( - \log^{\alpha} g\right)$ for any $\alpha < 1$ and for all $g$ sufficiently large. It does not seem easy to obtain an upper bound for $m'$ of the form $g^{1-\epsilon}$ for some $\epsilon > 0$. \end{remark}

By the Cauchy interlacing theorem (Theorem \ref{t:cauchy}) there holds
\begin{equation}\label{e:cauchyresult}
m\leq m'+\text{rank}(\text{Id}-P).
\end{equation}
Under the assumptions of Theorem \ref{t:volumebound} we can choose the $r_1$-net in a way that $\text{rank}(\text{Id}-P)\leq C_0 {\rm vol}(M)/r_1$ according to Lemma \ref{l:taillernet}, which together with \eqref{e:mm'} and \eqref{e:cauchyresult} concludes the proof in this first case for ${\rm vol}(M)$ large enough. Combining with \cite[Corollary 1.1]{hassan} we get the result for any ${\rm vol}(M)$.

Under the additional assumption that $\lambda_2(M)\geq \delta$, we can choose the $r_1$-net in a way that $\text{rank}(\text{Id}-P)\leq C_0 {\rm vol}(M)/e^{\delta' r_1}$ according to Lemma \ref{l:spectralgap} (with $\delta'$ given in this lemma), which together with \eqref{e:mm'} and \eqref{e:cauchyresult} proves the following statement:
\begin{theorem}\label{t:volboundspecgap}
For any $\rho>0$, $b<0$ and $\delta>0$ there exist $C_0,\alpha>0$ such that for any closed, connected Riemannian surface $M$ with ${\rm inj}(M)\geq \rho$, Gaussian curvature $\geq b$ and spectral gap $\lambda_2(M)\geq \delta$, the multiplicity of $\lambda_2(M)$ is at most $C_0(1+\frac{{\rm vol}(M)}{\log^\alpha (3+{\rm vol}(M))})$.
\end{theorem}
The first part of Theorem \ref{t:sublinearA} follows from Theorem \ref{t:volumebound} together with the Gauss-Bonnet theorem, which implies that ${\rm vol}(M)\leq \frac{4\pi}{|a|}(g-1)$. Similarly, the second part of Theorem \ref{t:sublinearA} follows from Theorem \ref{t:volboundspecgap} together with the Gauss-Bonnet theorem.

\begin{remark}\label{r:depC0}
To make the constants $C_0, C_1$ and $\alpha$ in Theorem \ref{t:sublinearA} explicit, we first multiply the Riemannian metric on $M$ by $|b|^{\frac12} + \rho^{-1}$ to obtain $M'$. Then $M' \in \mathcal{M}^{(a',-1,1)}$ where $a' = a\left(|b|^{\frac12} + \rho^{-1}\right)^{-2}$. Combining Theorem \ref{t:volumebound} (resp. Theorem \ref{t:volboundspecgap}) applied to $M'$ with parameters $b=-1$ and $\rho=1$ (resp. $b=-1$, $\rho=1$ and $\frac{\delta}{\left(|b|^{\frac12}+ \rho^{-1}\right)^2}$) and the Gauss--Bonnet formula we get that $C_0,C_1,\alpha$ may be taken as $$C_0 = C_1 = C_u\frac{|b|+\rho^{-2}}{|a|} \text{ and }\alpha = c_u\frac{\max\left(\frac{\sqrt{\delta}}{\sqrt{20}}, \frac{\delta}{4\sqrt{|b|}}\right)}{|b|^{\frac12}+\rho^{-1}}$$
where $C_u > 0 $ and $c_u > 0$ are two universal constants computable - in principle - from our methods.
\end{remark}

\begin{remark}\label{r:choiceparameter}
One can draw from \eqref{e:grosC0foismachinbidulle} a justification for our choices of $r_1$ and $r_2$ as \eqref{e:defr1r2}. Indeed, our goal is to make \eqref{e:grosC0foismachinbidulle} sublinear in ${\rm vol}(M)$. For the term $\exp(2\cutchi r_2\sqrt{|b|})/\mu_2^{4r_2}$, this requires $r_2=O(\log {\rm vol}(M))$. At the heuristic level, beyond time $\log {\rm vol}(M)$, the heat kernel is spread almost uniformly over $M$ (whose typical diameter is of order $\log {\rm vol}(M)$ for negatively curved surfaces under the spectral gap assumption), and extracting any kind of information from its analysis becomes difficult.

In turn, the term $\frac{{\rm vol}(M)}{\mu_2^{2r_1}}\exp\left(-C_7\exp(-hr_1) \lfloor r_2/r_1\rfloor\right)$ requires $\exp(-hr_1) \lfloor r_2/r_1\rfloor\rightarrow 0$ as ${\rm vol}(M)\rightarrow +\infty$, in particular $r_1=O(\log(r_2))$. We need $r_1$ largest possible due to \eqref{e:cauchyresult} and the fact that $\text{rank}(\text{Id}-P)$ is a decreasing function of $r_1$ (see Lemma \ref{l:taillernet} and Lemma \ref{l:spectralgap}). This explains our choice of $r_1=\Theta( \log\log {\rm vol}(M))$ and $r_2=\Theta(\log {\rm vol}(M))$.

In particular, the term ${\rm rank}({\rm Id}-P)$ in the right-hand side of \eqref{e:cauchyresult} cannot be made smaller than $\frac{{\rm vol}(M)}{\log\log{\rm vol}(M)}$ with our arguments (or $\frac{{\rm vol}(M)}{\log^\alpha {\rm vol}(M)}$,  if the spectral gap assumption is made). And we notice that this term is precisely the one which one would need to improve in order to enhance the final bound on $m$, since the bound \eqref{e:mm'} on $m'$ is indeed already much better.
\end{remark}

\subsection{Proof of Theorem \ref{t:extension}}\label{s:proofofextension}
Theorem \ref{t:extension} is a direct consequence of the Gauss-Bonnet formula together with the following result, which we prove in this section through elementary modifications of Section \ref{s:proofofth}.
\begin{theorem}\label{t:volumebound2} For any $j\in\mathbb{N}_{\geq 2}$, any $\rho,K,\beta>0$ and $b<0$ there exist $C_0, v_0>0$ such that for any closed, connected Riemannian surface $M$ with ${\rm inj}(M)\geq \rho$, ${\rm vol}(M)\geq v_0$, and Gaussian curvature $\geq b$, the number of eigenvalues in $[\lambda_j(M),(1+\frac{K}{\log^\beta {\rm vol}(M)})\lambda_j(M)]$ is at most $C_0(1+\frac{{\rm vol}(M)}{\log\log (3+{\rm vol}(M))})$.
\end{theorem}
\begin{proof}
     We need the following straightforward adaptation of Lemma \ref{l:defofS}.
\begin{lemma}\label{l:adaptdefofS}
For any $j\in\N_{\geq 2}$, there exists $C_j'>0$ and a subset $S\subset M$ of area $\nu(S)\leq C_j'\exp(3\cutchi r_2\sqrt{|b|})$ such that for any $x\notin S$, 
\begin{equation*}
\|e^{r_1\Delta}|\varphi_x|\|^2\leq \mu_j^{2r_1}  +C_j'\exp(-\cutchi r_2).
\end{equation*}
where $\varphi_x$ has been introduced in Section \ref{s:applicationsminmax}.
\end{lemma}

Fix $j\in\N_{\geq 1}$ and $\beta,K>0$. We denote by $m'$ the number of eigenvalues of $Pe^{r_1\Delta}P$ contained in $[\mu_j^{r_1}(1-\delta),\mu_j^{r_1}]$ where $\delta=K\frac{\log\log {\rm vol}(M)}{\log^{\beta}{\rm vol}(M)}$. Compared to \eqref{e:choicecutchi}-\eqref{e:choicec}, the constants $\cutchi$ and $\cutheat$ are fixed using $C_j$ (coming from Lemma \ref{l:encadrement}) instead of $C_2$, and \eqref{e:choicec} is replaced by 
\begin{equation*}
\frac{\beta}{4}\geq \left(3\cutchi\sqrt{|b|} - 4\log C_j+25(1+|b|)(1+\cutheat)\right) c.
\end{equation*}
Instead of \eqref{e:summarym'} we obtain using Lemma \ref{l:adaptdefofS}
\begin{align*}
m'\mu_j^{2r_1(\lfloor r_2/r_1\rfloor+1)}(1-\delta)^{2\lfloor r_2/r_1\rfloor+2}& \leq C_0 {\rm vol}(M)\left(\mu_j^{r_1}\left(1-\frac{C_6}{2}\exp(-hr_1)\right)\right)^{2\lfloor r_2/r_1\rfloor}+C_0\exp(2\cutchi r_2\sqrt{|b|})\\
&\qquad \qquad +C_0\exp\left(-\frac{\cutchi^2 r_2}{32}\right){\rm vol}(M).
\end{align*}
Dividing by $\mu_j^{2r_1(\lfloor r_2/r_1\rfloor+1)}(1-\delta)^{2\lfloor r_2/r_1\rfloor+2}$ and proceeding as in Section \ref{s:proofofth}, we obtain instead of \eqref{e:mm'}
$$
m'\leq C_0 {\rm vol}(M)\exp\left( - \log^{1-\frac{\beta}{2}} {\rm vol}(M)\right) (1-\delta)^{-2\lfloor r_2/r_1\rfloor-2}
$$
and thanks to the definition of $\delta$ and the inequality $(1-\delta)^n\leq e^{-n\delta}$, we finally get for sufficiently large $g$
$$
m'\leq C_0 g\exp\left((4K\log^{1-\beta}{\rm vol}(M))-(\log^{1-\frac{\beta}{2}}{\rm vol}(M))\right)\leq C_0\frac{{\rm vol}(M)}{\log\log {\rm vol}(M)}.
$$
By the Cauchy interlacing theorem (Theorem \ref{t:cauchy}) we obtain that the number $m$ of eigenvalues of $e^{r_1\Delta}$ in $[\mu_j^{r_1}(1-\delta),\mu_j^{r_1}]$ is bounded above by $C_0\frac{{\rm vol}(M)}{\log\log {\rm vol}(M)}$. It implies the same bound for the number of eigenvalues of $e^{\Delta}$ in $[\mu_j(1-\frac{K}{2c\log^{\beta}{\rm vol}(M)}),\mu_j]$, and Theorem \ref{t:volumebound2} follows. 
\end{proof}

Theorem \ref{t:extension} follows directly from Theorem \ref{t:volumebound2} together with the Gauss-Bonnet formula which implies that ${\rm vol}(M)\leq \frac{4\pi}{|a|}g$ for $M\in\mathcal{M}_g^{(a,b,\rho)}$.

\begin{remark} \label{r:tracemethod}
In the present paper, we rely on the trace method to bound eigenvalue multiplicity. The natural time scale of the trace which we consider, namely $(Pe^{r_1\Delta}P)^{\lfloor r_2/r_1\rfloor+1}\approx e^{\lfloor r_2/r_1\rfloor r_1\Delta}$,  is $O(r_1\lfloor r_2/r_1\rfloor)=O(c\log{\rm vol}(M))$. With this time scale, it is impossible to distinguish eigenvalues that differ by $O(1/ \log{\rm vol}(M))$. Analogously, the spectral bounds obtained in \cite[Theorems 4 and 5]{laura} do not give precise information in spectral windows of size $\ll 1/\sqrt{\log(g)}$.

\end{remark}

\subsection{Scale-free version of Theorem \ref{t:volumebound}} \label{s:scalefree}
We conclude this section with a version of Theorem \ref{t:volumebound} which involves only quantities which are invariant under rescaling of the metric, in the spirit of \cite[Corollary 1.1]{hassan}. For a closed connected Riemannian surface $M$, we define $\kappa(M)$ as the smallest $\kappa\geq 0$ such that $c(x)\geq -\kappa$ for any $x\in M$. We set $$
\sf={\rm vol}(M)(\kappa(M)+{\rm inj}(M)^{-2})
$$ 
which is a scale-free quantity, meaning that if the metric on $M$ is multiplied by a factor $R>0$, $\sf$ remains unchanged.
\begin{theorem}\label{t:sublinearA2}
There exists $C_{0}>0$ such that for any closed, connected Riemannian surface $M$, the multiplicity of $\lambda_2(M)$ is at most  $C_0(1+\frac{\sf}{\log\log(3+\sf)})$.
\end{theorem}
Theorem \ref{t:sublinearA2} improves (for surfaces) over the bound (1.14) in \cite{hassan}, which is linear in $\sf$. It is possible to prove scale-free bounds similar to Theorem \ref{t:sublinearA2} which generalize the second part of Theorem \ref{t:sublinearA} (with spectral gap assumption), and Theorems \ref{t:extension} and \ref{t:volumebound2}.
\begin{proof}[Proof of Theorem \ref{t:sublinearA2}]
Let $M$ be a closed connected Riemannian surface. Denote by $M_R$ the surface obtained by multiplying the metric on $M$ by $R>0$. For some $R_0\leq \max(\kappa(M)^{1/2},{\rm inj}(M)^{-1})$, we have ${\rm inj}(M_R)\geq 1$ and $\kappa(M_R)\leq 1$ for any $R\geq R_0$. Applying Theorem \ref{t:volumebound} to $M_{R_0}$, we obtain that the multiplicity of $\lambda_2(M_{R_0})$ is $\leq C_0(1+\frac{{\rm vol}(M_{R_0})}{\log\log (3+{\rm vol}(M_{R_0}))})$, and the same bound holds for the multiplicity of $\lambda_2(M)$ since multiplicity is preserved under scaling. Since ${\rm vol}(M_{R_0})\leq G(M)$, we get the result.
\end{proof}

\subsection{Proof of Proposition \ref{p:construction}}\label{s:construction}

Our proof of Proposition \ref{p:construction} essentially relies on the following lemma, extracted from \cite{ColboisCdV}.
\begin{lemma}[Extracted from \cite{ColboisCdV}]\label{l:cdvcolbois}
Let $G=(V,E)$ be a  non-oriented finite graph, possibly with loops and multiedges, whose vertices have degrees $d_i\geq 3$ for any $i\in V$, and whose edge lengths are denoted by $(\theta_{i,j})_{\{i,j\}\in E}$. Then there exists  a sequence of closed hyperbolic surfaces $(X^\varepsilon)_{\varepsilon>0}$ of genus $|E|-|V|+1$ whose first $|V|$ eigenvalues of the positive Laplacian $\lambda_1(\varepsilon)\leq \ldots\leq \lambda_{|V|}(\varepsilon)$ (repeated according to multiplicities) satisfy $\lambda_j(\varepsilon)=\varepsilon \zeta_j+O(\varepsilon^2)$ where $\zeta_1\leq \ldots\leq \zeta_{|V|}$ are the $|V|$ eigenvalues of the quadratic form
\begin{equation}\label{e:qtheta}
q_\theta(x)=\frac{1}{\pi}\sum_{\{i,j\}\in E}\theta_{i,j}|x_i-x_j|^2, \qquad x\in\R^V
\end{equation}
 on $L^2(V,\mu)$ and $\mu=2\pi\sum_{i\in V} (d_i-2)\delta_i$ with $\delta_i$ the Dirac mass on $i\in V$.
\end{lemma}
\begin{proof}[Sketch of proof of Lemma \ref{l:cdvcolbois} extracted from \cite{ColboisCdV}]
For any $i\in V$, we denote by $V_i$ the multiset of $j\in V$ such that $\{i,j\}\in E$ (the fact that $V_i$ is a multiset comes from the fact that we allow loops and multiedges). The degree of $i\in V$ is $d_i=|V_i|\geq 3$.

The authors of \cite{ColboisCdV} first construct a closed hyperbolic surface $X$ as follows: to the vertex $i\in V$ is associated $X_i$, a compact hyperbolic surface with $d_i$ free geodesics $(\gamma_{i,j})_{j\in V_i}$ on its boundary, by gluing $d_i-2$ pants (see \cite[Section VI]{ColboisCdV} and its figures for the case of the complete graph). This is done in a way that the length $\ell(\gamma_{i,j})$ is equal to $\theta_{i,j}$. To construct the surface $X$ we glue the pieces $X_i$ as indicated by the graph $G$: for $\{i,i'\}\in E$, we glue $X_i$ and $X_{i'}$ by identifying $\gamma_{i,i'}$ with $\gamma_{i',i}$ without twist. In particular if $i=i'$, i.e. the edge $\{i,i'\}$ is a loop, we identify without twist one $\gamma_{i,i}$ with another $\gamma_{i,i}$.

In \cite[Section II]{ColboisCdV}, the authors construct from $X$ a family of closed hyperbolic surfaces $X^\varepsilon$ ($0<\varepsilon\leq 1$) as follows. The geodesics in the pant decomposition of $X$ which do not belong to the boundary of one of the $X_i$, $i\in V$, remain of fixed length. For $\{i,j\}\in E$, the geodesic $\gamma_{i,j}$ of $X$ is replaced in $X^\varepsilon$ by a geodesic $\gamma^\varepsilon_{i,j}$ of length $\ell^\varepsilon_{i,j}=\varepsilon \theta_{i,j}$. Note that $\vol(X_i^\varepsilon)=\vol(X_i)=2\pi(d_i-2)$ for any $i\in V$, by the Gauss-Bonnet formula.

%In our case we choose all pants to have their three boundaries of length $1$, in particular all $\gamma_{i,j}$ have length $1$.

Then, in \cite[Section V]{ColboisCdV}, the authors consider the measure $\mu=2\pi\sum_{i\in V} (d_i-2)\delta_i$ on $G$, and the quadratic form $q_\theta$ on $L^2(V,\mu)$ given  by \eqref{e:qtheta},
which is the Dirichlet form on $G$ endowed with edge lengths $\theta=(\theta_{i,j})_{\{i,j\}\in E}$.
Then in \cite[Sections I and V]{ColboisCdV} they  exhibit a quadratic form $q_\theta^\varepsilon$ on $L^2(V,\mu)$ (depending continuously on the geometric parameter $\theta$) whose spectrum is the set of first $|V|$ eigenvalues of $X^\varepsilon$ and such that $\lim_{\varepsilon\rightarrow 0}\|(q_\theta^\varepsilon/\varepsilon)-q_\theta\|=0$, uniformly in $\theta\in W$ for every compact $W\Subset (\R_{>0})^E$. In particular it implies that the eigenvalues $(\lambda_i(\varepsilon))_{i\in V}$ of $X^\varepsilon$ verify $\lambda_i(\varepsilon)\sim \varepsilon\zeta_i$ ($\varepsilon\rightarrow 0$) where the $(\zeta_i)_{i\in V}$ are the eigenvalues of $q_\theta$ on $L^2(V,\mu)$.

The genus of $X^\varepsilon$ is equal to $(p+2)/2$ where $p=\sum_{i\in V} (d_i-2)=2|E|-2|V|$ is the number of pants used in the decomposition. This concludes the proof of Lemma \ref{l:cdvcolbois}. 
\end{proof}
\begin{proof}[End of the proof of Proposition \ref{p:construction}]
Let $n\in\mathbb{N}_{\geq 3}$. We consider the star graph $F_n$ with $n$ branches (i.e. $n$ edges and $n+1$ vertices). Since this graph has leaves (vertices of degree $1$), we cannot apply Lemma \ref{l:cdvcolbois} directly to $F_n$. Therefore, we also consider $G_n$ the graph obtained by adding a loop at each of the $n$ leaves of $F_n$. The central vertex of $G_n$ has degree $n$, and all $n$ other vertices of $G_n$ have degree $3$. All edges of $F_n$ and $G_n$ have length $1$. We denote by $V_n$ the vertex set of $G_n$, and by $E_n$ its multiset of edges. We have $|V_n|=n+1$ and $|E_n|=2n$. Finally, we set $\mu_n=2\pi\sum_{i\in V_n}(d_i-2)\delta_i$.

The eigenvalues of the quadratic form $q_\theta$ given by \eqref{e:qtheta} (for $G_n$) on $L^2(V_n,\mu_n)$ are equal to those of the (positive) Laplacian on $G_n$ given by
$$
(\Delta x)_i=\frac{1}{d_i-2}\sum_{j\in V_i}x_i-x_j
$$
(see e.g. \cite[Section 4]{ColArnold}). Its eigenvalues are $\frac{2n-2}{n-2}$, $1$ and $0$, with respective multiplicities $1, n-1$ and $1$.

For a given $n$, we use Lemma \ref{l:cdvcolbois} for sufficiently small $\varepsilon_n$. We obtain a closed connected hyperbolic surface $M_{g_n}$ of genus 
\begin{equation}\label{e:gendelasurface}
g_n=|E_n|-|V_n|+1=n.
\end{equation}
having at least $n-1$ eigenvalues in $[\lambda_2(M_{g_n}), (1+C_n\varepsilon_n)\lambda_2(M_{g_n})]$ for some $C_n>0$ which does not depend on $\varepsilon_n$. Taking $\varepsilon_n<\varepsilon_{g_n}/C_n$, this concludes the proof.
\end{proof}

\appendix 
\section{Appendix}\label{a:elementary}
We gather in this appendix several statements and proofs of elementary facts that are used throughout the proof of our main results.

\subsection{Eigenvalues and trace}\label{a:eigenvalues}
We first prove an infinite-dimensional version of the Cauchy interlacing theorem (see also \cite[Theorem 2]{dancis}).
\begin{theorem}[Cauchy interlacing theorem] \label{t:cauchy}
Let $A$ be a positive compact self-adjoint operator on a Hilbert space $H$. Let $P=P^\top$ be an orthogonal projection onto a subspace of $H$ of codimension $k\in\mathbb{N}$. We denote by $\alpha_1\geq \alpha_2\geq \ldots$ the eigenvalues of $A$,  and by  $\beta_1\geq \beta_2\geq \ldots$ those of $B=PAP$. Then for any $j\in\N$,
$$
\alpha_j\geq \beta_j\geq \alpha_{j+k}.
$$
\end{theorem}
\begin{proof}[Proof of Theorem \ref{t:cauchy}]
Since $B$ is compact and self-adjoint, the spectral theorem provides a basis $(b_j)_{j\in\mathbb{N}}$ of normalized eigenvectors of $B$, with $Bb_j=\beta_j b_j$ for any $j\in\N$ (notice that to order the $\beta_j$ we use that $B\ge 0$). We set $S_j=\text{Span}(b_1,\ldots,b_j)$ and we notice that $S_j\subset\text{Im}(P)$. We compute
\begin{align*}
\beta_j&=\underset{x\in S_j,\ \|x\|=1}{\min}(PAP x,x)=\underset{x\in S_j,\ \|x\|=1}{\min}(A x,x)\\
&\leq \underset{V, \ \text{dim}(V)=j }{\max}\ \ \underset{x\in V,\ \|x\|=1}{\min}(A x,x)=\alpha_j.
\end{align*}
Also, noticing that $P S_{j-1}^\perp$ has codimension at most $k+j-1$ we obtain
\begin{align*}
\beta_j&=\underset{x\in S_{j-1}^\perp,\ \|x\|=1}{\max}(PAP x,x)\geq \underset{x\in P S_{j-1}^\perp,\ \|x\|=1}{\max}(PAPx,x)\\
&=\underset{x\in P S_{j-1}^\perp,\ \|x\|=1}{\max}(Ax,x)
\geq \underset{V,\ \text{codim V}\leq k+j-1}{\min}\ \ \underset{x\in V,\ \|x\|=1}{\max}(Ax,x)= \alpha_{k+j}
\end{align*}
which concludes the proof.
\end{proof}
We recall the following estimate:
\begin{lemma}[Upper bound on eigenvalues] \label{l:encadrement}
For any $b\in\R$ and any $j\in\N_{\geq 2}$, there exists $C_j>0$ such that any closed surface $M$ with curvature bounded below by $b$ verifies  $\lambda_j(M)\leq C_j$. 
\end{lemma}
\begin{proof} 
The diameter $d$ of a closed surface $M$ with curvature bounded below by $b$ is bounded below since for any $x\in\widetilde{M}$,
$$
\frac{4\pi}{|b|}\leq \vol(M)\leq \Vol_{\hyp}(B_{\hyp}(x,d))\leq \frac{4\pi}{|b|}\sinh^2\left(\frac{d}{2}\sqrt{|b|}\right)
$$
where the left-hand side comes from Gauss-Bonnet and the right-hand side from \eqref{e:volumeballmodel}.
Combining with \cite[Corollary 2.3]{cheng75} we get the result.
\end{proof}

\begin{lemma}[Computation of the trace]\label{l:tracedelta}
For any $n\in\mathbb{N}_{\geq 1}$ and $t\geq 1$, there holds
$$
{\rm Tr}((Pe^{t\Delta}P)^{2n})=\int_M \|(Pe^{t\Delta}P)^n\delta_x\|^2 d\nu(x).
$$
\end{lemma}
\begin{proof}
We set $Q=Pe^{t\Delta}P$. Let $(u_j)_{j\in\mathbb{N}}$ denote an orthonormal basis of eigenfunctions of the compact and self-adjoint operator $Q^n$, with associated eigenvalues $\zeta_j$. For any $x\in M$ we set
$u_x=\sum_{j\in\N} \zeta_ju_j(x)u_j$.
We know that $Q^{2n}$ is trace-class since $e^\Delta$ is trace-class, and
\begin{equation}\label{e:exprtrace}
{\rm Tr}(Q^{2n})=\sum_{j\in\N} \zeta_j^2=\int_M\left(\sum_{j\in\N} \zeta_j^2u_j(x)^2\right)\nu(dx)=\int_M\|u_x\|^2\nu(dx).
\end{equation}
In particular, $u_x\in L^2(M,\nu)$  for $\nu$-almost every $x\in M$. For any such $x$ and  any $f\in C^\infty(M)$, written as $f=\sum_{j\in\N}a_ju_j$, we have
\begin{align}
\langle Q^n\delta_x,f\rangle_{\mathcal{D}',\mathcal{D}}&=\langle P\delta_x,e^{t\Delta}PQ^{n-1}f\rangle_{\mathcal{D}',\mathcal{D}}\nonumber=e^{t\Delta}PQ^{n-1}f(x)-\sum_{k\in\net} \langle e^{t\Delta}PQ^{n-1}f,\psi_k\rangle \psi_k(x)\nonumber\\
&= Q^nf(x)=\sum_{j\in\N} a_j\zeta_ju_j(x)=\int_Mu_x(y)f(y)\nu(dy) \nonumber
\end{align}
where the first equality comes from the fact that the transpose (in the sense of distributions) of the continuous linear map from smooth functions to smooth functions $e^{t\Delta}PQ^{n-1}$ is $Q^{n-1}Pe^{t\Delta}$; and the second equality follows from \eqref{e:Pdeltax}.
We deduce from this computation that $Q^n\delta_x$ coincides with the distribution $\langle u_x,\cdot\rangle_{L^2(M,\nu)}$, which is identified to $u_x\in L^2(M,\nu)$. Plugging into \eqref{e:exprtrace}, this concludes the proof.
\end{proof}

\subsection{Heat kernel: comparison and estimates}\label{a:heatkernel}
We provide here the proofs of Lemma \ref{l:heatkernelesthyp} and \ref{l:heatkernelinM} on the heat kernel in $\hyp$ and $M$.
\begin{proof}[Proof of Lemma \ref{l:heatkernelesthyp}]

We recall from \cite[Theorem 5.7.2]{davies} that there exist constants $c_1,c_2>0$ such that for any $\eta, t>0$,
\begin{equation*}
c_1 g_1(t,\eta)\leq k_t^{\mathbb{H}^2}(\eta)%\leq c_2 g_1(t,\eta)
\end{equation*}
where $k^{\mathbb{H}^2}$ denotes the heat kernel in the hyperbolic plane (equal to $k^{(-1)}$ with the notation of Lemma \ref{l:comparisonheat}) and
$$
g_1(t,\eta)=\frac{1}{t}\frac{1+\eta}{(1+\eta+t)^{\frac12}}\exp\left(-\frac{t}{4}-\frac{\eta}{2}-\frac{\eta^2}{4t}\right).
$$
For $K<0$ we consider 
$$
g_{|K|}(t,\eta)=|K|g_1(|K|t,|K|^{\frac12}\eta)=\frac{1}{t}\frac{1+|K|^{\frac12}\eta}{(1+|K|^{\frac12}\eta+|K|t)^{\frac12}}\exp\left(-\frac{|K|t}{4}-\frac{|K|^{\frac12}\eta}{2}-\frac{\eta^2}{4t}\right),
$$
which is the analogue of $g_1$ on the space form $\widetilde{M}_K$ introduced in Section \ref{s:notation}. Using Lemma \ref{l:comparisonheat} we obtain for the heat kernel $k_t(\cdot,\cdot)$ in $\hyp$ that 
\begin{equation}\label{e:fundaheat}
c_1g_{|b|}(t,d_{\hyp}(x,y))\leq k_t(x,y)
\end{equation}
for any $x,y\in\hyp$ and any $t>0$. Combining \cite[Theorem 3]{davies2} and \cite[Proposition 14]{croke} we also get the bound
\begin{equation}\label{e:fundaheat2}
k_t(x,y) \leq C_0\frac{1}{t}\left(1+\frac{d_{\widetilde{M}}(x,y)^2}{t}\right)\exp\left( -\frac{d_{\widetilde{M}}(x,y)^2}{4t}\right)
\end{equation} 
where $C_0>0$ depends on $b$ and $\rho$.
%(with the recent \cite[Theorem 1.6.(i)]{lizhang} suggesting $C_0 \leq {\rho^2}\left(|b| + \left(\frac{\eta}{t}\right)^2\right)$).

For \eqref{e:L1norm}, we set for $n\in\N$
$$
A_n=\left\{y\in\hyp \mid Ct+n\leq d_{\hyp}(x,y)<Ct+ n+1\right\}\subset \hyp.
$$
Then $\Vol_{\hyp}(A_n)\leq \Vol_{\hyp}(B(x,Ct+n+1))\leq C_0 \exp((Ct+n)\sqrt{|b|})$ according to \eqref{e:volumeballmodel}.
We write $\hyp\setminus B_{\hyp}(x,\cutheat t)=\bigcup_{n=0}^\infty A_n$, and then using \eqref{e:fundaheat2} and the fact that $C,t\geq 1$, we obtain 
\begin{align}
\|k_t(x,\cdot)\|_{L^1(\hyp\setminus B_{\hyp}(x,\cutheat t))}&\leq C_3 \sum_{n=0}^\infty \frac{(\cutheat t+n)^2}{t} \exp\left(-\frac{(\cutheat t+n)^2}{4t}\right)\Vol_{\hyp}(A_n)\nonumber\\
&\leq C_3\int_{\cutheat t-1}^\infty  \frac{\eta^2}{t}\exp\left(-\frac{\eta^2}{4t}\right)\exp(\eta\sqrt{|b|})d\eta \nonumber\\
&\leq C_3\exp(|b|t)\int_{\cutheat t-1}^\infty \frac{\eta^2}{t}\exp\left(-\frac{(\eta-2t\sqrt{|b|})^2}{4t}\right)d\eta \label{e:change}
\end{align}
where $C_3=\frac{C_0}{\rho^2}(1+|b|^2)$ for some universal constant $C_0 >0$.
We make the change of variables $\eta'=\eta-2t\sqrt{|b|}$ and we use that $\cutheat t-1 - 2t\sqrt{|b|}\geq 3\cutheat t/4$ and $\eta+2t\sqrt{|b|}\leq 2\eta$ for $\eta\geq 3Ct/4$ to obtain that \eqref{e:change} is bounded above by
$$C_3\exp(|b|t)\int_{3\cutheat t/4}^\infty \frac{\eta^2}{t}\exp\left(-\frac{\eta^2}{4t}\right)d\eta.$$ 
Computing this last integral gives the result. 

%We prove \eqref{e:heatlowerbound}. If $0\leq \eta\leq 2$, \eqref{e:heatlowerbound} may be checked by hand. If $\eta\geq 2$, there exists a universal constant $C_3>0$ such that\begin{equation}\label{e:petitcalcul}g_t(\eta)\geq C_3\frac{\eta}{t^{3/2}}\exp\left(-\frac{\eta}{2}-\frac{t}{2}\right).\end{equation} The function $\eta\mapsto \eta\exp(-\eta/2)$ is a decreasing function of $\eta\geq 2$, and \eqref{e:petitcalcul} becomes for $\eta=t$ exactly \eqref{e:heatlowerbound}. This concludes the proof of \eqref{e:heatlowerbound}.
%we get \begin{align*}K_t(x,\cdot)=\sum_{\gamma\in\Gamma}k_t(\bar{x},\gamma\bar{y})\end{align*} where $\bar{x}, \bar{y}$ are lifts of $x,y$ to a fixed fundamental domain in $\hyp$. It is classical that for any $\bar{x}\in\hyp$, the number of elements $\gamma\in\Gamma$ such that $d_{\hyp}(\bar{x},\gamma\bar{x})<R_0+1$ is at most $C_\rho e^{R_0}$ for some constant $C_\rho>0$ depending only on $\rho$ (see for instance \cite[Lemma 16]{laura}, which also gives the dependence of $C_\rho$ in $\rho$). As a consequence, for any $\bar{x},\bar{y}\in\hyp$ and $\eta\geq 0$, \begin{equation}\label{e:cardinal}\#\{\gamma\in\Gamma \mid \eta\leq d(\bar{x},\gamma\bar{y})<\eta+1\}\leq C_\rho e^{2\eta+2}.\end{equation}
For \eqref{e:variationsalpha}, we set $\eta=d_{\hyp}(x,y)$ and $\alpha=d_{\hyp}(x,z)-d_{\hyp}(x,y)$. We have, using again \eqref{e:fundaheat} and \eqref{e:fundaheat2},
\begin{equation}\label{e:ktalpha2}
\frac{k_t(x,z)}{k_t(x,y)} \geq C_4\frac{tg_{|b|}(t,\eta+\alpha)}{\left( 1 + \frac{\eta^2}{t}\right)\exp\left( -\frac{\eta^2}{4t}\right)}=C_4\frac{t\left(1+|b|^{\frac12}(\eta+\alpha)\right)}{\left(1 + \frac{\eta^2}{t}\right)\left(1+|b|^{\frac12}(\eta+\alpha)+|b|t\right)^{\frac12}}h(\alpha,\eta)
\end{equation}
where 
\begin{align}
h(\alpha,\eta)&=\exp\left(-\frac{|b|t}{4}-\frac{|b|^{\frac12}(\eta+\alpha)}{2}-\frac{(\eta+\alpha)^2}{4t}+\frac{\eta^2}{4t}\right)\nonumber \\
&\geq C_5\exp\left(-\frac{|b|t}{4}-\frac{|b|^{\frac12}(\cutheat +4)t}{2}-2\cutheat (t+1)-4t\right)\nonumber\\
%&\geq \exp\left(-\frac{|b|t}{4}-\frac{|b|^{\frac12}(\cutheat +4)t}{2}-2\cutheat (t+1)-4t\right)\nonumber \\
&\geq C_5\exp\left(-\frac{|b|t}{4}-\frac{(1+|b|)(\cutheat +4)t}{4}-4(C+1)t\right)\label{e:upfaleph2}
\end{align}
where we used in the second line $\eta\leq Ct$ and $|\alpha|\leq 4t +4$, and in the last line that $(1+|b|) \geq 2|b|^{\frac12}$ and $t\geq 1$.
Combining \eqref{e:ktalpha2} and \eqref{e:upfaleph2}, and using $\eta\leq Ct$ and $|\alpha|\leq 4t +4$ again, we get \eqref{e:variationsalpha}.
\end{proof}

\begin{proof}[Proof of Lemma \ref{l:heatkernelinM}]
We write $M=\Gamma\backslash\hyp$. We prove that there exists $C_0>0$ universal such that for any $\eta\geq 0$ and any  $\bar{x}\in\hyp$, the number of elements $\gamma\in\Gamma$ such that $d_{\hyp}(\bar{x},\gamma\bar{x})<\eta+1$ is at most $C_0 \rho^{-2}e^{\eta\sqrt{|b|}}$.  By definition of the injectivity radius $\rho$, the open balls $B_\gamma$ of center $\gamma\bar{x}$ and radius $\rho/2$, for $\gamma\in\Gamma$, are disjoint. If $\gamma\in\Gamma$ is such that $d_{\hyp}(\bar{x},\gamma \bar{x})< \eta+1$, then $B_\gamma$ is included in the ball of center $\bar{x}$ and radius $\eta+1+\rho/2$. According to \eqref{e:volumeballmodel}, the volume of a ball of radius $\eta+1+\rho/2$ in $M$ is at most $\frac{4\pi}{|b|}\sinh^2(\frac12 (\eta+1+\rho/2)\sqrt{|b|})$, and according to \cite[Proposition 14]{croke}, the volume of a ball of radius $\rho/2$ is at least $C_1\rho^2>0$. Therefore, the number of $\gamma\in\Gamma$ such that $d_{\hyp}(\bar{x},\gamma\bar{x})<\eta+1$ is smaller than 
$$
\frac{C_0}{\rho^2|b|}\sinh^2\left(\frac12\left(\eta+1+\frac{\rho}{2}\right)\sqrt{|b|}\right)
$$
which in turn is bounded above by $C_0|b|^{-1}\rho^{-2}e^{\eta\sqrt{|b|}}$.

As a consequence, for any $\bar{x},\bar{y}\in\hyp$ and $\eta\geq 0$,
\begin{equation}\label{e:cardinal}
\#\{\gamma\in\Gamma \mid \eta\leq d(\bar{x},\gamma\bar{y})<\eta+1\}\leq \frac{C_0}{|b|\rho^2}e^{2\eta\sqrt{|b|}}.
\end{equation}
Below, $\bar{x}$, $\bar{y}$ are lifts of given $x,y\in M$ to a fundamental domain of $M$ in $\hyp$.
For any $y\in M$ we have, using \eqref{e:fundaheat2} in the first line and \eqref{e:cardinal} in the second line,  
\begin{align*}
\sum_{\gamma\in\Gamma} k_t(\bar{x},\gamma\bar{y})%&\leq \sum_{\eta=R}^\infty \sum_{\substack{\gamma\in\Gamma\\ \eta\leq d(\bar{x},\gamma\bar{y})<\eta+1}} k_1(\bar{x},\gamma\bar{y})\\
&\leq C_1 \sum_{\eta=0}^\infty \left(\#\{\gamma\in\Gamma \mid \eta\leq d(\bar{x},\gamma\bar{y})<\eta+1\}\right) \frac{1}{t}\left(1+\frac{\eta^2}{t}\right)\exp\left(-\frac{\eta^2}{4t}\right)\\
&\leq C_1 \sum_{\eta=0}^\infty\frac{1}{t}\left(1+\frac{\eta^2}{t}\right)\exp\left(2\eta\sqrt{|b|}-\frac{\eta^2}{4t}\right)
\end{align*}
where $C_1 > 0$ depends on $b$ and $\rho$. For any $t>0$ this sum converges. Using then a series-integral comparison for the last inequality (cutting the sum at $\eta=4t\sqrt{|b|}$) and \eqref{e:heatkernelinM} we get the result.
\end{proof}

\end{document}